\documentclass[dvips,ejs]{imsart}

\usepackage[latin1]{inputenc}
\usepackage{amsthm,amsmath,amssymb, amsfonts}
\usepackage{dsfont}
\usepackage{graphicx}
\RequirePackage[authoryear]{natbib}
\RequirePackage[colorlinks,citecolor=blue,urlcolor=blue]{hyperref}
\RequirePackage{hypernat}
\usepackage{shortcuts}

\doi{10.1214/154957804100000000}
\pubyear{0000}
\volume{0}
\firstpage{0}
\lastpage{0}

\startlocaldefs

\theoremstyle{plain}
\newtheorem{thm}{Theorem}[section]
\newtheorem{prop}[thm]{Proposition}

\newtheorem{lem}[thm]{Lemma}
\newtheorem{cor}[thm]{Corollary}
\newtheorem{assum}{Assumption}

\newtheorem{ex}{Example}

\newcommand{\Zdefn}{\mathcal{Z}_{n}}
\newcommand{\Xdefn}{\mathcal{X}_{n}}

\newcommand{\Xn}{X_{[n]}}

\newcommand{\Zn}{Z_{[n]}}
\newcommand{\zn}{z_{[n]}}
\newcommand{\taun}{\tau_{[n]}}
\renewcommand{\L}{\mathcal{L}}
\newcommand{\J}{\mathcal{J}}
\newcommand{\znh}{\widehat{z}_{[n]}}
\newcommand{\zh}{\widehat{z}}
\newcommand{\tah}{\widehat{\tau}}
\newcommand{\tahn}{\tah_{[n]}}
\newcommand{\znt}{\widetilde{z}_{[n]}}

\newcommand{\PXn}{P^{\Xn}}
\newcommand{\De}{\Delta(\pi,\znh,\znt)}
\newcommand{\Dp}{\Delta^+(\pi,\znh,\znt)}
\newcommand{\Dm}{\Delta^-(\pi,\znh,\znt)}
\renewcommand{\Dh}{\widehat D}
\newtheorem{postita}{Post-it}

\endlocaldefs

\begin{document}
\sloppy


\begin{frontmatter}
\title{Consistency of maximum-likelihood and variational estimators in the Stochastic Block Model \thanksref{t2}}
\thankstext{t2}{Authors thank Catherine Matias and Mahendra Mariadassou for their helpful comments.}
\runtitle{MLE and variational estimators in SBM}


\author{\fnms{Alain} \snm{Celisse}\thanksref{t1}
\corref{c1}
\ead[label=e1]{celisse@math.univ-lille1.fr}}
\thankstext{t1}{This work has been supported by the French Agence Nationale de la Recherche
(ANR) under reference ANR-09-JCJC-0027-01}
\address{Laboratoire de Math\'ematique \\
UMR 8524 CNRS-- Universit\'e Lille 1\\
F-59\,655 Villeneuve d'Ascq Cedex, France \\
 \printead{e1}}

\and

\author{\fnms{Jean-Jacques} \snm{Daudin}\ead[label=e2]{daudin@agroparistech.fr}}
\address{UMR 518 AgroParisTech/INRA MIA\\
16 rue Claude Bernard\\
F-75\,231 Paris Cedex 5, France \\
\printead{e2}}

\and
\author{\fnms{Laurent} \snm{Pierre}\ead[label=e3]{laurent.pierre@u-paris10.fr}}
\address{Universit\'e Paris X\\
200 avenue de la R\'epublique\\ 
F- 92\,001 Nanterre Cedex France\\
\printead{e3}}

\runauthor{Celisse, Daudin, Pierre}

\begin{abstract}
The \textit{stochastic block model} (SBM) is a probabilistic model designed to describe heterogeneous \textit{directed} and \textit{undirected} graphs.
In this paper, we address the asymptotic inference in SBM by use of maximum-likelihood and variational approaches.
The identifiability of SBM is proved while asymptotic properties of maximum-likelihood and variational estimators are derived.
In particular, the consistency of these estimators is settled for the probability of an edge between two vertices (and for the group proportions at the price of an additional assumption), which is to the best of our knowledge the first result of this type for variational estimators in random graphs.
\end{abstract}




\end{frontmatter}


\section{Introduction}

In the last decade, networks have arisen in numerous domains such as social sciences and biology.
They provide an attractive graphical representation  of complex data.
However, the increasing size of networks and their great number of connections have made it difficult to interpret
 {\it network representations of data} in a satisfactory way.
This has strengthened the need for statistical analysis of such networks, which could raise latent patterns in the data.

Interpreting networks as realizations of random graphs,
{\it unsupervised classification} (clustering) of the vertices of the graph has received much attention.
It is based on the idea that vertices with a similar connectivity can be gathered in the same class.
The initial graph can be replaced by a simpler one without loosing too much information.
This idea has been successfully applied to social  \citep{NS} and biological \citep{PMDCR} networks. It is out of the scope of the present work to review all of them.
%

Mixture models are a convenient and classical tool to perform unsupervised classification in usual statistical settings.
Mixture models for random graphs were first proposed by \citet{HLL} who defined the so-called {\it stochastic block} model (SBM),
in reference to an older {\it non stochastic block} model widely used in social science.
Assuming each vertex belongs to only one class, a {\it latent variable} (called the \textit{label}) assigns every vertex to its corresponding class.
SBM is a versatile means to infer underlying structures of the graph. Subsequently, several versions of SBM have been studied and it is necessary to formally distinguish between them.
Three binary distinctions can be made to this end:
\begin{enumerate}

\item The graph can be \textit{directed} or \textit{undirected}.

\item The graph can be \textit{binary} or \textit{weighted}.

\item The model can $(i)$ rely on \emph{latent random variables} (the labels), or $(ii)$ assume the labels are \emph{unknown parameters}:\\
$(i)$ SBM is a usual mixture model with random multinomial \textit{latent variables} \citep{NS,DPR,AM}.
In this model, vertices are sampled in a population and the concern is on the population parameters, that is the frequency of each class and their connectivity parameters.

$(ii)$ An alternative \emph{conditional} version of \emph{SBM} (called CSBM) has been studied \citep{BC}. 
In CSBM, former latent random variables (the labels) are considered as \textit{fixed parameters}. The main concerns are then the estimation of between- and within-class connectivity parameters as well as of the unknown label associated to every vertex \citep[see][]{RCY,CWA}.
\end{enumerate}
The present work deals with directed (and undirected) binary edges in random graphs drawn from SBM. 

The main interest of SBM is that it provides a \emph{more realistic and versatile model than the famous Erd\"{o}s-R\'enyi graph} while remaining easily interpretable.
However unlike usual statistical settings where independence is assumed, one specificity of SBM is that \emph{vertices are not independent}.
Numerous approaches have been developed to overcome this challenging problem, but most of them suffer some high computational cost.
For instance \citet{SN} study maximum-likelihood estimators of SBM with only two classes and  \textit{binary undirected} graphs, while \citet{NS} perform Gibbs sampling for more than two classes at the price of a large computational cost.
Other strategies also exist relying for instance on \emph{profile-likelihood optimization} in CSBM \citep{BC}, on a \emph{spectral clustering algorithm} in CSBM \citep{RCY}, or on \emph{moment estimation} in a particular instance of SBM called \emph{affiliation model} \citep[][and also Example~\ref{ex.affiliation.model} in the present paper]{AM}.

A \emph{variational approach} has been proposed by \citet{DPR} to remedy this computational burden. 
It can be used with \emph{binary directed} SBM and avoids the algorithmic complexity of the likelihood and Bayesian approaches (see \citet{Mix} and also \citet{MRV} for \textit{weighted undirected} SBM analyzed with a variational approach).
However even if its practical performance shows a great improvement,  \emph{variational approach remains poorly understood from a theoretical point of view}.
For instance, no consistency result does already exist for \textit{maximum likelihood} or \textit{variational} estimators of SBM parameters.
Note however that consistency results for maximum likelihood estimators in the CSBM have been derived recently by \citet{CWA} where the number of groups is allowed to grow with the number of vertices.  
Nonetheless, empirical clues \citep{GDR} have already supported the consistency of variational estimators in SBM.
Establishing such asymptotic properties is precisely the purpose of the present work.

\medskip

In this paper the identifiability of \emph{binary directed} SBM is proved under very mild assumptions for the first time to our knowledge.
Note that identifiability of directed SBM is really challenging
since existing strategies such as that of \cite{AMR} cannot be extended easily.
The \emph{asymptotics of maximum-likelihood} and \emph{variational estimators} is also addressed by use of concentration inequalities.
In particular, variational estimators are shown to be asymptotically equivalent to maximum-likelihood ones, and consistent for estimating the probability $\pi$ of an edge between two vertices. When estimating the group proportions $\alpha$, an additional assumption on the convergence rate of $\pih$ is required to derive consistency.
The present framework assumes the number $Q$ of classes to be known and independent of the number of vertices.
Some attempts exist to provide a data-driven choice of $Q$ \citep[see][]{DPR}, but this question is out of the scope of the present work.

The rest of the paper is organized as follows.
The main notation and assumptions are introduced in Section~\ref{sec.identifiability.SBM} where identifiability of SBM is settled.
Section~\ref{sec.MLE.SBM} is devoted to the consistency of the maximum-likelihood  estimators (MLE),
and Section~\ref{sec.variational} to the asymptotic equivalence between
variational and maximum-likelihood estimators.
In particular, the consistency of variational estimators (VE) is proved.
Some concluding remarks are provided in Section~\ref{sec.conclusion} with some further important questions.


\section{Model definition and identifiability}\label{sec.identifiability.SBM}

Let $\Omega=(\mathcal{V},\X)$ be the set of infinite random graphs
where $\mathcal{V}=\N$ denotes the set of countable vertices and
$\X=\acc{0,1}^{\N^2}$ the corresponding set of adjacency matrices.
The random adjacency matrix, denoted by $X=\acc{X_{i,j}}_{i,j \in \N}$, is given by: for $i \neq j$, $X_{i,j}=1$  if an edge exists from vertex $i$ to vertex $j$ and $X_{i,j}=0$ otherwise, and $X_{i,i}=0$ (no loop).
Let $\P$ denote a probability measure on $\Omega$.

\subsection{Stochastic Block Model (SBM)}\label{subsec.modelSBM}

Let us consider a random graph with $n$ vertices
$\acc{v_i}_{i=1,\ldots,n}$.
These vertices are assumed to be split into $Q$ classes
$\acc{C_q}_{q=1,\ldots,Q}$ depending on their structural
properties.

Set $\alpha = (\alpha_1,\ldots,\alpha_Q)$ with $0 < \alpha_q < 1$
and $\sum_q \alpha_q =1$.
For every $q$, $\alpha_q$ denotes the probability for a given
vertex to belong to the class $C_q$.
For any vertex $v_i$, its label $Z_i$ is generated as follows
\begin{align*}
     \acc{Z_i}_{1 \leq i \leq n} \  \stackrel{\iid}{\sim}  \mathcal{M}\paren{n; \alpha_1,\ldots,\alpha_Q}\enspace.
\end{align*}
where $\mathcal{M}\paren{n; \alpha_1,\ldots,\alpha_Q}$ denotes the
multinomial distribution.
Let $\Zn=\paren{Z_1,\ldots,Z_n}$ denote the random label vector of $(v_1,\ldots,v_n)$.

The observation consists of an adjacency matrix
$X_{[n]}=\acc{X_{i,j}}_{1\leq i,j\leq n}$, where $X_{i,i}=0$ for
every $i$ and
\begin{eqnarray*}
      X_{i,j} \mid Z_i=q, Z_j=l \   & \stackrel{\iid}{\sim} &    \mathcal{B}\paren{\pi_{q,l}},\quad \forall i\neq j\enspace,
\end{eqnarray*}
where $\mathcal{B}(\pi_{q,l})$ denotes the Bernoulli distribution
with parameter $0\leq \pi_{q,l}\leq 1$ for every $(q,l)$.

The log-likelihood is given by
\begin{align}\label{def.log-likelihood.SBM}
    \L_2(\Xn;\alpha,\pi) = \log\paren{ \sum_{\zn} e^{\L_1(\Xn;\zn,\pi)}\P\croch{\Zn=\zn} } \enspace,
\end{align}
where
\begin{align}\label{def.log-likelihood.L1}
    \L_1(\Xn;\zn,\pi) = \sum_{i\neq j} \acc{ X_{i,j}\log\pi_{z_i,z_j}+(1-X_{i,j}) \log(1-\pi_{z_i,z_j} ) } \enspace,
\end{align}
and $\P\croch{\Zn=\zn}=\P_{\alpha}\croch{\Zn=\zn}=\prod_{i=1}^n \alpha_{z_i}$.
In the following, let $\theta=(\alpha,\pi)$ denote the parameter
and $\theta^*=(\alpha^*,\pi^*)$ be the true parameter value.
Notice that the $X_{i,j}$s are not independent. However,
conditioning on $Z_i=q, Z_j=l$ yields independence.

Recall that the number $Q$ of classes is assumed to
be known and the purpose of the present work is to efficiently
estimate the parameters of SBM.

\subsection{Assumptions}

In the present section, several assumptions are discussed, which will be used all along the paper . 

\begin{assum}
For every $q\neq q'$, there exists $l \in \acc{1,\ldots,Q}$ such that
\begin{align}\label{assum.pi.block.diff} 
\pi_{q,l} \ne \pi_{q',l}\ \mathrm{or}\ \pi_{l,q} \ne \pi_{l,q'} \enspace. \tag{\bf A1}
\end{align}
\end{assum}
Following \eqref{assum.pi.block.diff}, the matrix $\pi$ cannot have two columns equal and the corresponding rows also equal.
This constraint is consistent with the goal of SBM which is to
define $Q$ classes $C_1,\ldots,C_Q$ with different structural
properties.
For instance, the connectivity properties of vertices in $C_q$
must be different from that of vertices in $C_l$ with $q\neq l$.
Therefore, settings where this assumption is violated correspond
to ill-specified models with too many classes.

\begin{assum}
There exists $\zeta>0$ such that
\begin{align}\label{assum.pi.trunc}
\forall (q,l) \in \acc{1,\ldots,Q}^2 ,\quad \pi_{q,l}\in]0,1[\quad \Rightarrow\quad \pi_{q,l}\in[\zeta,1-\zeta]\enspace. \tag{\bf A2}
\end{align}
\end{assum}
SBM can deal with null probabilities of connection between
vertices. However, the use of $\log\pi_{q,l}$ implies a special
treatment for $\pi_{q,l}\in\acc{0,1}$.  
Note that all along the present paper, \eqref{assum.pi.trunc} is always assumed to hold
with $\zeta$ not depending on $n$.

\begin{assum}
There exists $0<\gamma <1/Q$ such that
\begin{align}\label{assum.alpha.trunc}
    \forall q \in \acc{1,\ldots,Q},\quad \alpha_q\in [\gamma,1-\gamma]\enspace. \tag{\bf A3}
\end{align}

\end{assum}
This assumption implies that no class is drained. 
Actually the identifiability of SBM (Theorem~\ref{thm.identifiability1}) requires every $ \alpha_q\in(0,1)$ for $q\in\acc{1,\ldots,Q}$, which is implied by \eqref{assum.alpha.trunc}. 
In this paper, it is assumed that $\gamma$ does not
depend on $n$.

\begin{assum}
There exist $0<\gamma <1/Q$ and $n_0\in\N^*$ such that any realization of SBM (Section~\ref{subsec.modelSBM}) with label vector $\zn^*=(z^*_1,\ldots,z^*_n)$ satisfies
\begin{align}\label{assum.alpha.empirique.trunc}
    \forall q \in \acc{1,\ldots,Q}, \forall n\geq n_0,\quad \frac{N_q(\zn^*)}{n} \geq \gamma \enspace, \tag{\bf A4}
\end{align}
where $N_q(\zn^*) = \abs{\acc{1\leq i \leq n \mid z^*_i=q}}$.
\end{assum}
Note that \eqref{assum.alpha.empirique.trunc} is the empirical version of \eqref{assum.alpha.trunc}. 
By definition of SBM, $\zn^*$ is the realization of a multinomial random variable with parameters $(\alpha_1,\ldots,\alpha_Q)$.
Any multinomial random variable will satisfy
the requirement of \eqref{assum.alpha.empirique.trunc} with high probability.
This assumption will be used in particular in Theorem~\ref{thm.distrib.conv.zn}.

\subsection{Identifiability}\label{sec.identifiability}

The identifiability of the parameters in SBM have been first
obtained by \citet{AMR} for {\it undirected} graphs ($\pi$ is
symmetric): if $Q=2$, $n \ge 16$, and the coefficients of $\pi$
are all different, the parameters are identifiable up to label
switching.
\citet{AMR1} also established that for $Q>2$, if $n$ is even and
$n \geq (Q-1+ \frac{(Q+2)^2}{4})^2$ (with a similar condition if
$n$ is odd), the parameters of SBM are generically identifiable,
that is identifiable except on a set with null Lebesgue measure.

First generic identifiability (up to label
switching) of the SBM parameters is proved for {\emph directed} (or
undirected) graphs as long as $n\geq 2Q$.

\begin{thm}\label{thm.identifiability1}
Let $n\geq 2Q$ and assume that for any $1\leq q \leq Q$, $\alpha_q
>0$ and the coordinates of $r=\pi . \alpha$ are distinct.
Then, SBM parameters are identifiable.

\end{thm}

The assumption on vector $\pi . \alpha$ is not strongly restrictive since the set of
vectors violating this assumption is of Lebesgue measure 0. Therefore,
Theorem~\ref{thm.identifiability1} actually asserts the
generic identifiability of SBM \citep[see][]{AMR}.
Moreover, Theorem~\ref{thm.identifiability1} also holds with $r'=\pi^{\,t}.\alpha$ (instead of $r=\pi. \alpha$), and
with vectors $r''$ given by $r_q''=\sum_l\pi_{q,l}\pi_{l,q}\alpha_l$ for
every $1\leq q\leq Q$.
Let us also emphasize that Assumption~\eqref{assum.pi.block.diff} is 
implied by assuming either $\pi.\alpha$ or $\pi^t.\alpha$ has distinct coordinates (which leads to identifiability).
Note that \citet[Theorem 2, Section~3.1]{2011_BCL} also recently derived an identifiability result for ``block models'' in terms of ``wheel motifs''.

Let us further illustrate the assumption on $\pi.\alpha$ through two examples.
The first one is a particular instance of SBM called \emph{Affiliation Model} \citep{AMR1} restricted to the setting where $Q=2$.
\begin{ex}[Affiliation model] 
\label{ex.affiliation.model}
From a general point of view, affiliation model is used with $Q$ populations of vertices and considers undirected graphs ($\pi$ symmetric).
The present example focuses on a particular instance where $Q=2$.
In this model, the matrix $\pi$ is only parametrized by two coefficients $\pi_1$ and $\pi_2$ ($\pi_1\neq \pi_2$), which respectively correspond to within-class and between-class connectivities between edges.
With $Q=2$, the matrix $\pi$ is given by 
\[ \pi =  
\paren{\begin{array}{cc}
  \pi_1 & \pi_2 \\
\pi_2 & \pi_1 
\end{array}} \enspace.
\]
Then, requiring $\paren{\pi\alpha}_1=\pi_1\alpha_1+\pi_2\alpha_2$  is not equal to $\paren{\pi\alpha}_2=\pi_2\alpha_1+\pi_1\alpha_2$ amounts to impose that $\alpha_1\neq \alpha_2$.
Indeed since within- and between-class connectivities are the same for the two classes, distinguishing between them therefore requires a different proportion of edges in these classes ($\alpha_1\neq \alpha_2$).

Note that  \citet{AMR1} have derived a result on identifiability for affiliation models with equal group proportions.
\end{ex}

The second example describes a more general setting than Example~\ref{ex.affiliation.model} in which the assumption on the coordinates of $r$ can be more deeply understood.
\begin{ex}[Permutation-invariant matrices]\label{ex.permutation.invariant}
For some matrices $\pi$, there exist permutations $\sigma:$ $\acc{1,\ldots,Q} \to \acc{1,\ldots,Q}$ such that 
$\pi$ remains unchanged if one permutes both its rows and 
columns according to $\sigma$.
More precisely, let $\pi^{\sigma}$ denote the matrix defined by 
\begin{align*}
  \pi^{\sigma}_{q,l} = \pi_{\sigma(q),\sigma(l)} \enspace,
\end{align*}
for every $1\leq q,l \leq Q$.
Then, $\pi^{\sigma}=\pi$.

For a given matrix $\pi$, let us define the set of permutations letting $\pi$ invariant by
\begin{align*}
\mathfrak{S}^{\pi} = \acc{ \sigma:\ \acc{1,\ldots,Q} \to \acc{1,\ldots,Q}\mid \pi^{\sigma}=\pi }\enspace.  
\end{align*}
The matrix \emph{$\pi$ is said permutation-invariant} if $\mathfrak{S}^{\pi}\neq \acc{Id}$, where $Id$ denotes the identity permutation.
For instance in the affiliation model (Example~\ref{ex.affiliation.model}), $\pi$ is permutation-invariant since $\mathfrak{S}^{\pi}$ is the whole set of permutations on $\acc{1,\ldots,Q}$.

Let us first notice that ``label-switching'' translates into the following property.
For any permutation $\sigma$ of $\acc{1,\ldots,Q}$, 
\begin{align}\label{eq.label.switching}
  \pi^{\sigma}\alpha^{\sigma} = \pi \alpha \enspace,
\end{align}
where $\alpha^{\sigma}_q = \alpha_{\sigma(q)}$ for every $q$.
The main point is that label-switching arises whatever the choice of $(\alpha,\pi)$, and for every $\sigma$.

By contrast, only permutation-invariant matrices satisfy the more specific following equality. 
For any permutation-invariant matrix $\pi$, let $\sigma^{\pi}\in \mathfrak{S}^{\pi}$ denote one permutation whose support is of maximum cardinality. (Such a permutation is not necessarily unique, for instance with the affiliation model.)
Then, 
\begin{align} \label{eq.permutation.invariant}
  (\pi\alpha)^{\sigma^{\pi}} = \pi \alpha \enspace.
\end{align}
Equation~\eqref{eq.permutation.invariant} amounts to impose equalities of the coordinates of $\pi\alpha$ in the support of $\sigma^{\pi}$. 
Let us recall that the support of $\sigma^{\pi}$ corresponds to rows and columns of $\pi$ that can be permuted without changing $\pi$.
Then, assuming all coordinates of $\pi\alpha$ distinct leads to impose that 
classes with the same connectivity properties have different respective proportions ($\alpha_q$) to be distinguished between one another.
\end{ex}

\smallskip

\begin{proof}[Proof of Theorem~\ref{thm.identifiability1}]
First, let $P_{[n]}$ denote the probability distribution function of the
adjacency matrix $\Xn$ of SBM.
Let us show that there exists a unique
$(\alpha,\pi)$  corresponding to $P_{[n]}$.

\medskip
Up to reordering, let $ r_1 < r_2 < \ldots < r_Q $ denote the
coordinates of the vector $r$ in the increasing order:
$r_q$ is equal to the probability of an edge from a given vertex
in the class $C_q$.

Let $R$ denote the Van der Monde matrix defined by
$R_{i,q}=r_q^i$, for $0\leq i<Q$ and $1\leq q \leq Q$.
$R$ is invertible since the coordinates of $r$ are all different.
For $i\geq1$, $R_{i,q}$ is the probability that $i$ given vertices
in $C_q$ have an edge.

Let us also define
\begin{align*}
u_i=\sum_{1\leq k\leq Q} \alpha_k r_k^i,\quad i=0,\ldots,2Q-1\enspace.
\end{align*}
For $i \geq 1$, $u_i$ denotes the probability that the first $i$
coefficients of the first row of $\Xn$ are equal to 1.
Note that $n\geq 2Q$ is a necessary requirement on $n$ since  $X_{i,i}=0$ by assumption.
Hence given $P_{[n]}$, $u_0=1$ and $u_1,\ldots,u_{2Q-1}$ are
known.

Furthermore, set $M$ the $(Q+1) \times Q$ matrix given by
$M_{i,j}=u_{i+j}$ for every $0\leq i\leq Q$ and $0\leq j<Q$, and
let $M_i$ denote the square matrix obtained by removing the row
$i$ from $M$.
The coefficients of $M_Q$ are
\begin{align*}
M_{i,j}=u_{i+j}=\sum_{1\leq k \leq Q } r_k^i \alpha_k
r_k^j\enspace,\quad \mathrm{with}\quad 0 \leq  i,j < Q\enspace.
\end{align*}
Defining the diagonal matrix $A=\mathrm{Diag}(\alpha)$, it comes
that $M_Q=RAR^{\,t}$, where $R$ and $A$ are invertible, but
unknown at this stage.
With $D_k=\det(M_k)$ and the polynomial $B(x)=\sum_{k=0}^Q
(-1)^{k+Q} D_k\, x^k$, it yields $D_Q=\det(M_Q)\neq0$ and the
degree of $B$ is equal to $Q$.

\medskip
Set $V_i=(1,r_i,\ldots,r_i^Q)^t$ and let us notice that $B(r_i)$
is the determinant of the square matrix produced when appending
$V_i$ as last column to $M$.
The $Q+1$ columns of this matrix are linearly dependent, since
they are all linear combinations of the $Q$ vectors $V_1$, $V_2$,
$\ldots$, $V_Q$.
Hence $B(r_i)=0$ and $r_i$ is a root of $B$ for every $1\leq i\leq
Q$.
This proves that $B=D_Q\prod_{i=1}^Q(x-r_i)$.
Then, one knows $r=(r_1,\ldots,r_Q)$ (as the roots of $B$ defined
from $M$) and $R$.
It results that $A= R^{-1} M_Q \paren{R^{\,t}}^{-1}$, which yields
a unique $(\alpha_1,\ldots,\alpha_Q)$.

\medskip

It only remains to determine $\pi$.
For $0 \leq i,j <Q $, let us introduce $U_{i,j}$ the probability that
the first row of
$\Xn$ begins with $i+1$ occurrences of 1, and the second row of $X$
ends up with $j$ occurrences of 1 ($i+1+j\leq n-1$ implies $n\geq 2Q$).

Then, $U_{i,j}=\sum_{k,l} r_k^i \alpha_k \pi_{k,l} \alpha_l
r_l^j$, for $0\leq i,j < Q$, and
the $Q \times Q$ matrix $U = RA\pi A R^{\,t}$.
The conclusion results from
$\pi = A^{-1}R^{-1}U {(R^{\,t})}^{-1}A^{-1}$.

\end{proof}

The assumption of Theorem~\ref{thm.identifiability1} on $r$ ($r'$ or $r''$),
leading to {\em generic identifiability}, can be further relaxed in the particular
case where $n = 4$ and $Q=2$.
\begin{thm}\label{thm.identifiability2}
Set $n = 4$, $Q=2$ and let us assume that $\alpha_q >0$ for every
$1\leq q \leq Q$, and the coefficients of $\pi$ are not all equal.
Then, SBM is identifiable.
\end{thm}
The proof of this result is deferred to Appendix~\ref{appendix.identif.2}.

Note that when $Q=2$, \eqref{assum.pi.block.diff} implies the coefficients of $\pi$ are not all equal.



\section{Maximum-likelihood estimation of SBM parameters}\label{sec.MLE.SBM}


\subsection{ Asymptotics of $\P \paren{ \Zn=\cdot \mid \Xn }$}

In this section we study the \textit{a posteriori} probability distribution function of $\Zn$, $\P\paren{\Zn=\cdot\mid \Xn}$,
which is a random variable depending on $\Xn$.

\subsubsection{Equivalence classes between label sequences}\label{sec.equiv.class}
Let us consider a realization of the SBM random graph generated with the sequence of true labels $Z=z^*$, where $z^*=\acc{z^*_i}_{ i \in \N^* }$.

Since a given matrix $\pi$ can be \emph{permutation-invariant} (see Example~\ref{ex.permutation.invariant} Section~\ref{sec.identifiability}), the mapping $z \mapsto \acc{\pi_{z_i,z_j}}_{i,j\in\N^*}$ can be non injective.
To remedy this problem, let us introduce an equivalence relation between two sequences of labels $z$ and $z'$:
\begin{align*}
  z \stackrel{\pi}{\sim} z' \quad \Leftrightarrow \quad \exists \sigma\in\mathfrak{S}^{^\pi} \mid \ z'_i=\sigma(z_i),\quad \forall i\in\N^*\enspace.
\end{align*}
Then $z \stackrel{\pi}{\sim} z'$ is equivalent to $\croch{z}_{\pi} = \croch{z'}_{\pi}$, where $\croch{z}_{\pi}$ denotes the equivalence class of $z$. 
As a consequence, any vectors $\zn$ and $\zn'$ in the same class have the same conditional likelihood \eqref{def.log-likelihood.L1}:
\begin{align*}
 \L_1(\Xn;\zn,\pi) =  \L_1(\Xn;\zn',\pi) \enspace.  
\end{align*}

From now on, square-brackets in the equivalence class notation will be removed to simplify the reading as long as no confusion can be made. In such cases, $z$ will be understood as the equivalence class of the label sequence.

\subsubsection{Main asymptotic result} \label{subsec.main.asymptotic.result}
Let $P^*\defegal\P\paren{\cdot\mid Z=z^*}=P_{\alpha^*,\pi^*}^*$ denote the \emph{true} conditional distribution given the (equivalence class of the) whole label sequence, the notation emphasizing that $P^*$ depends on $(\alpha^*,\pi^*)$.

The following Theorem~\ref{thm.distrib.conv.zn} provides the convergence rate of $\P\paren{\Zn=\zn^* \mid \Xn} = \P_{\alpha^*,\pi^*}\paren{\Zn=\zn^* \mid \Xn}$ towards 1 with respect to $P^*$, that is given  $Z=z^*$.
It is an important result that will be repeatedly used along the paper.

\begin{thm}\label{thm.distrib.conv.zn}
Let us assume that assumptions \eqref{assum.pi.block.diff}--\eqref{assum.alpha.empirique.trunc} hold. For every $t>0$,
\begin{align*}
P^*\croch{  \sum_{\zn\neq \zn^*}\frac{\P\paren{\Zn=\zn\mid \Xn}}{\P\paren{\Zn=\zn^*\mid \Xn}} > t } = \mathcal{O}\paren{n e^{-\kappa n}}\enspace,
\end{align*}
where $\kappa>0$ is a constant depending on $\pi^*$ but not on $z^*$, and the $\mathcal{O}\paren{n e^{-\kappa n}}$ is uniform with respect to $z^*$.

Furthermore, the same result holds with $P^*$ replaced by $\P$ under \eqref{assum.pi.block.diff}--\eqref{assum.alpha.trunc}.
\end{thm}

The proof of Theorem~\ref{thm.distrib.conv.zn} is deferred to Appendix~\ref{append.conv.loi.Z}.

A noticeable feature of this result is that the convergence rate does not depend on
$z^*$. This point turns out to be crucial when deriving consistency for the MLE and the variational estimator (respectively Section~\ref{sec.consistency.M1} and Section~\ref{subsubsec.uniform.asymptotics.J.L}).
Besides, the exponential bound of Theorem~\ref{thm.distrib.conv.zn} allows the use of  Borel-Cantelli's lemma to get the $\P-$almost sure convergence.
\begin{cor} \label{cor.distrib.conv.as.zn} With the same notation as Theorem
\ref{thm.distrib.conv.zn},
\begin{align*}
\sum_{\zn\neq \zn^*}\frac{\P\paren{\Zn=\zn\mid \Xn}}{\P\paren{\Zn=\zn^*\mid \Xn}}\xrightarrow[n\to+\infty]{}0\enspace,\qquad \P-a.s.\enspace.
\end{align*}
Moreover,
\begin{align*}
    \P\paren{\Zn=\zn^*\mid \Xn} \xrightarrow[n\to+\infty]{} 1 \enspace,\qquad \P-a.s.\enspace,
\end{align*}
and for every $\zn\neq \zn^*$,
\begin{align*}
    \P\paren{\Zn=\zn \mid \Xn} \xrightarrow[n\to+\infty]{} 0 \enspace,\qquad \P-a.s.\enspace.
\end{align*}
\end{cor}

As a consequence of previous Corollary~\ref{cor.distrib.conv.as.zn}, 
one can also understand the above phenomenon in terms of the conditional 
distribution of the equivalence class $\Zn$ given $\Xn$.
\begin{cor}\label{cor.degeneracy.distrib.zn}
\begin{align*}
    \mathcal{D}(\Zn\mid \Xn) \xrightarrow[n\to+\infty]{w} \delta_{z^*},\quad \P-a.s.\enspace,
\end{align*}
where $\mathcal{D}(\Zn\mid \Xn)$  denotes the distribution of
$\Zn$ given $\Xn$, $\xrightarrow[n\to+\infty]{w}$ refers to the
weak convergence in $\mathcal{M}_1\paren{\mathcal{Z}}$, the set of
probability measures on $\mathcal{E}\paren{\mathcal{Z}}$ the set of equivalence classes on $\mathcal{Z}=\acc{1,\ldots,Q}^{\N^*}$ and
$\delta_{z^*}$ is the Dirac measure at the equivalence class $z^*$.
\end{cor}

\smallskip

\begin{proof}[Proof of Corollary~\ref{cor.degeneracy.distrib.zn}]
For every $n\in\N^*$, let us define  $\Z_n=\acc{1,\ldots,Q}^n$ and $\mathcal{E} \paren{\Z_n}$ the corresponding set of equivalence classes.
Let us further introduce a metric space $(\mathcal{E}\paren{\Z_n},d_n)$, where the distance $d_n$ is given by
\begin{align*}
    \forall z,z'\in \mathcal{E}\paren{\Z_n},\quad d_n\paren{z,z'} = \min_{u\in z,\, v\in z'} \sum_{k=1}^n 2^{-k} \1_{\paren{u_k\neq v_k}}\enspace.
\end{align*}
Similarly for $\mathcal{Z}=\acc{1,\ldots,Q}^{\N^*}$, $(\mathcal{E}\paren{\mathcal{Z}},d)$ denotes a metric space with
\begin{align*}
    \forall z,z'\in \mathcal{E}\paren{\mathcal{Z}},\quad d\paren{z,z'} = \min_{u\in z,\, v\in z'}\sum_{k \geq 1} 2^{-k} \1_{\paren{u_k\neq v_k}}\enspace.
\end{align*}
Then, $\mathcal{E}\paren{\Z_n}$ can be embedded into $\mathcal{E}\paren{\mathcal{Z}}$, so that
$\mathcal{E}\paren{\Z_n}$ is identified to a subset of $\mathcal{E}\paren{\mathcal{Z}}$.

Let us introduce $\mathcal{B}$ the Borel $\sigma-$algebra on $\mathcal{E}\paren{\mathcal{Z}}$,
and $\mathcal{B}_n$ the $\sigma-$algebra induced by $\mathcal{B}$
on $\mathcal{E}\paren{\Z_n}$.
Let also $\P^n=\P\croch{\cdot\mid \Xn}$ denote a probability
measure on  $\mathcal{B}$, and $\E_n\croch{\cdot}$ is the
expectation with respect to $\P^n$.

Set $h\in C_b\paren{\Z}$ (continuous bounded functions on $\mathcal{E}\paren{\mathcal{Z}}$)
such that $\norm{h}_{\infty}\leq M$ for $M>0$.
By continuity at point $z^*$, for every $\epsilon>0$, there exists $\eta>0$ such that
\begin{align*}
d(z,z^*) \leq \eta \quad \Rightarrow \quad \abs{h(z^*)-h(z)} \leq
\epsilon \enspace.
\end{align*}
Then,
\begin{align*}
\abs{\E_n\croch{h\paren{\Zn}}-h(z^*)} \leq & \sum_{\zn}\abs{h\paren{\zn}-h(z^*)}\P^n\paren{\Zn=\zn} \\
 \leq & \ \epsilon + 2M \sum_{ \zn\in (B_{\eta}^*)^c} \P^n\paren{\Zn=\zn} \\
 \leq &\ \epsilon + o_{\P}(1) \quad \P-a.s.\enspace,
\end{align*}
by use of Corollary~\ref{cor.distrib.conv.as.zn}, where $B_{\eta}^*=B(z^*,\eta)$ denotes the ball in $\mathcal{E}\paren{\mathcal{Z}}$ with radius $\eta$ with respect to $d$.
In the last inequality, $o_{\P}(1)$ results from Corollary~\ref{cor.distrib.conv.as.zn}, which yields the result.

\end{proof}


\subsection{MLE consistency}
\label{sec.consistency.M1}
The main focus of this section is to settle the consistency of the MLE of $(\alpha^*,\pi^*)$.

Let us start by recalling the SBM log-likelihood \eqref{def.log-likelihood.SBM}:
\begin{align*}
    \L_2(\Xn;\alpha,\pi) = \log\paren{ \sum_{\zn} e^{\L_1(\Xn;\zn,\pi)}\P\croch{\Zn=\zn} } \enspace,
\end{align*}
where  $\P\croch{\Zn=\zn}=\prod_{i=1}^n \alpha_{z_i}$, and $(\alpha,\pi)$ are the SBM parameters.
Note that $\L_2(\Xn;\alpha,\pi)$ is an involved expression to deal with. 

First, the $X_{i,j}$s are not independent, which strongly differs from usual statistical settings.
For this reason, no theoretical result has ever been derived for the MLE of SBM parameters. 

Second, another non standard feature of $\L_1$ is the number of random variables which is $n(n-1)$ (and not $n$ as usual).
More precisely, there are $n(n-1)$ edges $X_{i,j}$s but only $n$ vertices. 
This unusual scaling difference
implies a refined treatment of the normalizing constants $n$ and $n(n-1)$, depending on the estimated parameter $\alpha$ and $\pi$ respectively. 
As a consequence, the MLE consistency proof has been split into two parts: ($i$) the consistency of the $\pi$ estimator is addressed by use of an approach based on M-estimators, ($ii$) a result similar to Theorem~\ref{thm.distrib.conv.zn} is combined with a ``deconditioning'' argument to get the consistency of the $\alpha^*$ estimator (Theorem~\ref{thm.consist.alpha.MLE}) at the price of an additional assumption on the rate of convergence of the estimator $\pih$ of $\pi^*$.

\medskip

The consistency of the MLE of $\pi$ strongly relies on a general theorem which is inspired from that for M-estimators \citep{VW}.
\begin{thm} \label{thm.Mestim.general}
Let $\paren{\Theta,d}$ and  $\paren{\Psi,d'}$ denote metric spaces,
and let $M_n:\ \Theta\times \Psi\to \R$ be a random function and $\mathbb{M}:\ \Theta \to \R$ a deterministic one such that for every $\epsilon>0$,
\begin{align}
&   \sup_{ d\paren{ \theta, \theta_0 } \geq \epsilon } \mathbb{M}\paren{\theta } <  \mathbb{M}\paren{\theta_0 }\enspace, \label{ineg.max.Bien.Identif.MLE}\\
&   \sup_{  \paren{\theta,\psi} \in \Theta\times\Psi } \abs{M_n\paren{\theta,\psi}-\mathbb{M}\paren{\theta}} \defegal \norm{M_n-\mathbb{M}}_{\Theta\times\Psi} \xrightarrow[n\to +\infty]{P} 0 \enspace. \label{ineg.conv.unif.MLE}
\end{align}
Moreover, set $( \widehat{\theta},\widehat{\psi} ) = \argmax_{\theta, \psi}M_n\paren{\theta,\psi}$.
Then,
\begin{align*}
    d\paren{\widehat{\theta},\theta_0} \xrightarrow[n\to +\infty]{P} 0 \enspace.
\end{align*}
\end{thm}
One important difference between Theorem~\ref{thm.Mestim.general} and its usual counterpart for M-estimators \citep{VW} is that $M_n$ and $\mathbb{M}$ do not depend on the same number of arguments.
Our consistency result for the MLE of $\pi$ strongly relies on this point.

\begin{proof}[Proof of Theorem~\ref{thm.Mestim.general}]
For every $\eta>0$, there exists $\delta>0$ such that
\begin{align*}
    P\croch{ d\paren{\widehat{\theta},\theta_0} \geq \eta} & \leq P\croch{ \mathbb{M}(\widehat{\theta}) \leq \mathbb{M}(\theta_0) - 3\delta}\enspace.
\end{align*}
Since $\norm{M_n-\mathbb{M}}_{\Theta\times\Psi} \xrightarrow[n\to +\infty]{P} 0$, it comes that for large enough values of $n$,
\begin{align*}
    P\croch{ d\paren{\widehat{\theta},\theta_0} \geq \eta} & \leq P\croch{ M_n(\widehat{\theta},\widehat{\psi}) \leq M_n(\theta_0,\psi_0) - \delta} + o(1)\\
& \leq o(1) \enspace.
\end{align*}

\end{proof}
The leading idea in what follows is to check the assumptions of Theorem~\ref{thm.Mestim.general}.

\medskip

The main point of our approach consists in using $P^*=P^*_{\alpha^*,\pi^*}$ (Section~\ref{subsec.main.asymptotic.result}) as a reference probability measure, that is working as if $\Zn=\zn^*$ were known.
In this setting, a key quantity is
\begin{align*}
    \L_1(\Xn;\zn,\pi)=\sum_{i\neq j}
\acc{ X_{i,j}\log\pi_{z_i,z_j}+(1-X_{i,j}) \log(1-\pi_{z_i,z_j} ) } \enspace,
\end{align*}
where $(\zn,\pi)$ are interpreted as parameters.
For any  $(\zn,\pi)$, let us introduce
\begin{align*}
    \phi_n\paren{\zn,\pi} & \defegal \frac{1}{n(n-1)} \L_1\paren{\Xn;\zn,\pi}\enspace, \\
    \Phi_n\paren{\zn,\pi} & \defegal \E\croch{\phi_n\paren{\zn,\pi} \mid \Zn=\zn^*}\enspace,
\end{align*}
where the expectation is computed with respect to $P^*=P^*_{\alpha^*,\pi^*}$. 
Actually our strategy (using Theorem~\ref{thm.Mestim.general}) only requires to prove $\phi_n$ and $\Phi_n$ are uniformly close to each other on a subset of parameters denoted by $\mathcal{P}$ (see also the proof of Theorem~\ref{thm.unif.conv} for more details) and defined as follows
\begin{align} \label{eq.admissible.set.parameters}
    \mathcal{P}=\acc{ (\zn,\pi) \mid \eqref{assum.pi.block.diff},\ \eqref{assum.pi.trunc},\ \abs{\Phi_n\paren{\zn,\pi}}<+\infty } \enspace.
\end{align}
Showing this uniform convergence between $\phi_n$ and $\Phi_n$ over $\mathcal{P}$ is precisely the purpose of Proposition~\ref{prop.unif.conv.cond.model}.
Its proof, which is deferred to Appendix~\ref{Appendix.prop.unif.conv.M1}, strongly relies on Talagrand's concentration inequality \citep{Mass_2007}.
\begin{prop}\label{prop.unif.conv.cond.model}
With the above notation, let us assume \eqref{assum.pi.block.diff} and \eqref{assum.pi.trunc} hold true. Then,
\begin{align*}      \sup_{ \mathcal{P} } \abs{\phi_n(\zn,\pi)-\Phi_n(\zn,\pi)} \xrightarrow[n\to +\infty]{\P} 0 \enspace.
\end{align*}
\end{prop}
Actually Proposition~\ref{prop.unif.conv.cond.model} is crucial to prove the following theorem which settles the desired properties for $\L_2(\Xn; \alpha,\pi)$, that is  \eqref{ineg.max.Bien.Identif.MLE} (uniform convergence) and  \eqref{ineg.conv.unif.MLE} (well-identifiability).

\begin{thm}\label{thm.unif.conv} Let us assume that  \eqref{assum.pi.block.diff}, \eqref{assum.pi.trunc}, and \eqref{assum.alpha.trunc} hold, and
for every $(\alpha,\pi)$, set
$   M_n(\alpha,\pi)  =  \croch{n(n-1)}^{-1}\L_2(\Xn;\alpha,\pi)\enspace$, and
\begin{align*}
& \mathbb{M}(\pi) \\
 =&  \max_{ \acc{a_{i,j}} \in\mathcal{A} }\acc{ \sum_{q,l}\alpha_q^* \alpha_l^* \sum_{q',l'}\croch{ a_{q,q'}a_{l,l'}\pi^*_{q,l} \log \pi_{q',l'}+(1-\pi^*_{q,l})
\log(1-\pi_{q',l'}) } }\enspace,
\end{align*}
where $(\alpha^*,\pi^*)$ denotes the true parameter of SBM, and
$\mathcal{A}  = \acc{ A=\paren{a_{i,j}}_{1\leq i,j \leq Q}\mid a_{q,q'} \geq 0,\ \sum_{q'} a_{q,q'}=1 } \subset \mathcal{M}_Q(\R)$.\\
Then for any $\eta>0$,
\begin{align*}
& \sup_{d(\pi,\pi^*) \ge \eta } \mathbb{M}(\pi)<\mathbb{M}(\pi^*)
 \enspace,\\ 
& \sup_{\alpha,\pi} \abs{ M_n(\alpha,\pi) - \mathbb{M}(\pi) } \xrightarrow[n\to+\infty]{\P} 0 \enspace,
\end{align*}
where $d$ denotes a distance.
\end{thm}
The proof of Theorem~\ref{thm.unif.conv} is given in Appendix~\ref{appendix.theorem.consist.MLE}.
Its uniform convergence part exploits the connection between $\phi_n(\zn,\pi)$ and $\L_2(\Xn; \alpha,\pi)$ (Proposition~\ref{prop.unif.conv.cond.model}).

Let us now deduce the Corollary~\ref{cor.consist.pi.MLE}, which asserts the consistency of the MLE of $\pi^*$.
\begin{cor}\label{cor.consist.pi.MLE}
    Under the same assumptions as Theorem~\ref{thm.unif.conv}, let us define the MLE of $(\alpha^*, \pi^*)$
\begin{align*}
    (\widehat{\alpha},\widehat{\pi}) \defegal \argmax_{(\alpha,\pi)} \L_2(\Xn; \alpha,\pi)\enspace.
\end{align*}
Then for any distance $d(\cdot,\cdot)$ on the set of parameters $\pi$,
\begin{align*}
    d\paren{\widehat{\pi},\pi^*} \xrightarrow[n\to +\infty]{\P} 0 \enspace.
\end{align*}

\end{cor}

\begin{proof}[Proof of Corollary~\ref{cor.consist.pi.MLE}]
    This is a straightforward consequence of Theorem~\ref{thm.Mestim.general} and Theorem~\ref{thm.unif.conv}.
\end{proof}

\medskip

A quick inspection of the proof of uniform convergence in Theorem~\ref{thm.unif.conv}
shows that the asymptotic behavior of the log-likelihood $\L_2$
does not depend on $\alpha$.
Roughly speaking, this results from the expression of $\L_2$ in which the number of terms involving $\pi$ is of order $n^2$ whereas only $n$ terms involve $\alpha$.
This difference of scaling with respect to $n$ between $\pi$ and $\alpha$ justifies to some extent a different approach for the MLE of $\alpha^*$.

Our proof heavily relies on  an analogous result to Theorem~\ref{thm.distrib.conv.zn}, where the true value $(\alpha^*,\pi^*)$ of SBM parameters is replaced by an estimator $(\widehat{\alpha},\widehat{\pi})$.
In what follows, $\widehat{P}\paren{\Zn=\zn\mid
\Xn}=\P_{\widehat{\alpha},\widehat{\pi}}\paren{\Zn=\zn\mid
\Xn}$ (Section~\ref{subsec.main.asymptotic.result} and Lemma~\ref{lem.mixture.model.class.freq}) denotes the same quantity as $\P\paren{\Zn=\zn\mid
\Xn}$ where $(\alpha^*,\pi^*)$ has been replaced by $(\widehat{\alpha},\pih)$.
Let us state this result in a general framework since it will be successively used in proofs of Theorems~\ref{thm.consist.alpha.MLE} and~\ref{thm.consist.alpha.VE}.
\begin{prop}
\label{prop.distrib.conv.zn.2}
Let us assume that assumptions \eqref{assum.pi.block.diff}--\eqref{assum.alpha.empirique.trunc} hold, and that there exists an estimator $\pih$ such that $\norm{\pih - \pi^*}_{\infty} = o_{\P}(v_n)$, with $v_n=o\paren{\sqrt{\log n}/n}$.
Let also $\widehat{\alpha}$ denote any estimator of $\alpha^*$.
Then for every $\epsilon>0$,
\begin{align*}
P^*\croch{ \sum_{\zn\neq \zn^*} \frac{ \widehat{P}\paren{\Zn=\zn\mid
\Xn} }{ \widehat{P}\paren{\Zn=\zn^*\mid \Xn} } > \epsilon } \leq
\kappa_1 n e^{-\kappa_2 \frac{(\log n)^2}{n v_n^2} } +  \P\croch{\norm{\widehat{\pi}-\pi^*}_{\infty}>v_n} \enspace
\end{align*}
for $n$ large enough, where $\kappa_1,\kappa_2>0$ are constants independent of $z^*$, and
\begin{align*}
& \log
\paren{ \frac{ \widehat{P}\paren{\Zn=\zn\mid
\Xn} }{ \widehat{P}\paren{\Zn=\zn^*\mid \Xn} } } \\
= &
   \sum_{i\neq j}
\acc{X_{i,j}\log\paren{\frac{\widehat{\pi}_{z_i,z_j}}{\widehat{\pi}_{z^*_i,z^*_j}}}
+(1-X_{i,j})\log\paren{\frac{1-\widehat{\pi}_{z_i,z_j}}{1-\widehat{\pi}_{z^*_i,z^*_j}}}}
+\sum_i\log\frac{\widehat{\alpha}_{z_i}}{\widehat{\alpha}_{z^*_i}} \enspace.
\end{align*}
Moreover, the same result holds replacing $P^*$ by $\P$ under \eqref{assum.pi.block.diff}--\eqref{assum.alpha.trunc}.
\end{prop}

The proof of Proposition~\ref{prop.distrib.conv.zn.2} is given in Appendix~\ref{append.prop.loi.Z}. 

In the same way as in Theorem~\ref{thm.distrib.conv.zn}, one crucial point in Proposition~\ref{prop.distrib.conv.zn.2} is the independence of the convergence rate with respect to $\zn^*$. 
Note that the novelty of Proposition~\ref{prop.distrib.conv.zn.2} compared to Theorem~\ref{thm.distrib.conv.zn}  lies in the convergence rate which depends on the behavior of $\widehat{\pi}$. This is the reliable price for estimating rather than knowing $\pi^*$.

We assume $v_n=o\paren{\sqrt{\log n}/n}$, which arises from the proof as a necessary requirement for consistency. However, we do not know whether this is a \emph{necessary} or only a \emph{sufficient} condition.
Furthermore there is still empirical evidence \citep[see][]{GDR} that the rate of convergence of  $\widehat{\pi}$ is of order $1/n$, but this property is assumed and not proved in the present paper.

\medskip

Let us now settle the consistency of the MLE of $\alpha^*$ on the basis of previous Proposition~\ref{prop.distrib.conv.zn.2}.

\begin{thm}\label{thm.consist.alpha.MLE}
Let $(\widehat{\alpha},\pih)$ denote the MLE of $(\alpha^*,\pi^*)$ and assume $\norm{\pih - \pi^*}_{\infty} = o_{\P}\paren{\sqrt{\log n}/n}$.
With the same assumptions as Theorem~\ref{thm.unif.conv}, and the notation of Corollary~\ref{cor.consist.pi.MLE}, then
\begin{align*}
    d( \widehat{\alpha}, \alpha^*) \xrightarrow[n \to +\infty]{\P} 0 \enspace,
\end{align*}
where $d$ denotes any distance between vectors in $\R^Q$.
\end{thm}
Note that the rate $1/n$ would be reached in ``classical'' parametric models with $n^2$ independent random variables.

\smallskip

\begin{proof}[Proof of Theorem~\ref{thm.consist.alpha.MLE}]
In the mixture model framework of SBM, Lemma~\ref{lem.mixture.model.class.freq} shows the MLE of $\alpha$ is given for any $q$ by
\begin{align*}
    \widehat{\alpha}_q = \frac{1}{n}\sum_{i=1}^n \widehat{P}(Z_i=q \mid \Xn )  \enspace.
\end{align*}
First, let us work with respect to $P^*$, that is given $\Zn=\zn^*$.
Setting $N_q(\zn)=\sum_{i=1}^n \1_{\paren{z_i=q}}$, it comes
\begin{align*}
    \abs{ \widehat{\alpha}_q - N_q(\zn^*)/n } \leq
& \abs{ \frac{1}{n} \sum_{i=1}^n \widehat{P}\paren{Z_i=z^*_i\mid \Xn} \1_{(z^*_i=q)} - N_q(\zn^*)/n }\\
& +
\widehat{P}\paren{Z_{[n]} \neq \zn^* \mid \Xn}\\
\leq &  \frac{1}{n} \sum_{i=1}^n \paren{ 1 - \widehat{P}\paren{Z_i=z^*_i\mid \Xn} }  \1_{(z^*_i=q)} \\
& +
\widehat{P}\paren{Z_{[n]} \neq \zn^* \mid \Xn}\\
\leq &  \frac{1}{n} \sum_{i=1}^n \paren{  \widehat{P}\paren{Z_i \neq z^*_i\mid \Xn} }  \1_{(z^*_i=q)} +
\widehat{P}\paren{Z_{[n]} \neq \zn^* \mid \Xn}\\
\leq & \ 2 \widehat{P}\paren{Z_{[n]} \neq \zn^* \mid \Xn}
\enspace.
\end{align*}
Note the last inequality results from
\begin{align*}
\frac{1}{n}\sum_i \widehat{P}\paren{Z_i \neq z^*_i\mid \Xn} & \leq \max_{i=1,\ldots,n} \widehat{P}\paren{Z_i \neq z^*_i \mid \Xn} \\
& \leq \widehat{P} \croch{ \cup_{i=1}^n (Z_i \neq z_i^*) \mid \Xn } = \widehat{P}\paren{Z_{[n]} \neq \zn^* \mid \Xn} \enspace.  
\end{align*}

Second, let us now use a ``deconditioning argument'' replacing $P^*$ by $\P$.
Let $N_q=N_q(\Zn)$ denote a binomial random variable $\mathcal{B}(n,\alpha^*_q)$ for every $q$.
Then for every $\epsilon>0$,
\begin{align*}
    &\  \P\croch{ \abs{\widehat{\alpha}_q-\alpha_q^*}>\epsilon } \\
 \leq &\
\P\croch{ \abs{ \widehat{\alpha}_q - N_q/n }>\epsilon/2 } +
\P\croch{ \abs{N_q/n-\alpha_q^*}>\epsilon/2 } \\
 \leq &\  \P\croch{ \abs{ \widehat{\alpha}_q - N_q/n }>\epsilon/2 } +
o(1) \enspace,
\end{align*}
by use of LLG.
Finally, a straightforward use of Proposition~\ref{prop.distrib.conv.zn.2} leads to
\begin{align*}
& \ \P \croch{ \abs{ \widehat{\alpha}_q - N_q/n }>\epsilon/2 } \\
= &\  \E_{\Zn}\croch{ P\paren{ \abs{ \widehat{\alpha}_q - N_q/n }>\epsilon/2 \mid Z_{[n]}} } \\
\leq & \  \sum_{\zn} P\croch{ \widehat{P}\paren{Z_{[n]} \neq \zn \mid \Xn} > \epsilon/4 \mid Z_{[n]}=\zn} \P\croch{Z_{[n]}=\zn} \\
= &\ o(1) \enspace.
\end{align*}

\end{proof}


\section{Variational estimators of SBM parameters}\label{sec.variational}

In Section~\ref{sec.MLE.SBM}, consistency has been proved for the maximum likelihood estimators. 
However this result is essentially theoretical since in practice the MLE can only be computed for very small graphs (with less than $20$ vertices). 
Nevertheless, such results for the MLE are useful in at least two respects. First from a general point of view, they provide a new strategy to derive  consistency of estimators obtained from likelihoods in non-\iid settings.
Second in the framework of the present paper, these results are exploited to settle the consistency of the variational estimators.

The main interest of \emph{variational estimators} in SBM is that unlike the MLE ones, they are \emph{useful in practice} since they enable to deal with very huge graphs (several thousands of vertices).
Indeed the log-likelihood
$\L_2\paren{\Xn; \alpha,\pi}$ involves a sum over $Q^n$ terms, which is intractable except for very small and unrealistic values of $n$:
\begin{align}\label{exp.likelihood}
\L_2(\Xn; \alpha,\pi) = \log \acc{  \sum_{\zn \in \Zdefn} e^{
\sum_{i \neq j} b_{ij}(z_i,z_j)  } P_{\Zn}(\zn) }    \enspace,
\end{align}
with $ b_{ij}(z_i,z_j)=X_{i,j}\log\pi_{z_i,z_j}+ (1-X_{i,j})\log
(1-\pi_{z_i,z_j})$.
To circumvent this problem, alternatives are for instance Markov chain Monte Carlo (MCMC) algorithms \citep{AnAt_2007} and variational approximation \citep{JGJS_1999}.
However, MCMC algorithms suffer a high computational cost, which makes them unattractive compared to variational approximation.
Actually the variational method can deal with thousands of vertices in a reasonable computation time thanks to its complexity in $\mathcal{O}(n^2)$. 
For instance, \citet{Mix} package (based on variational approximation) deals with up to several thousands of vertices, whereas STOCNET package \citep[see][]{Stocnet} (Gibbs sampling) only deals with a few hundreds of vertices.
Note that other approaches based on \emph{profile-likelihood} have been recently developed and studied for instance by \citet{BC}.

\smallskip

The purpose of the present section is to prove that the variational approximation
yields consistent estimators of the SBM parameters.
The resulting estimators will be called {\it variational estimators} (VE).

\subsection{Variational approximation} \label{subsec.variational.approx}
To the best of our knowledge, the first use of variational approximation for SBM has been made by \citet{DPR}.
The variational method consists in approximating $P^{\Xn}=\P\paren{\Zn=\cdot\mid \Xn}$
by a product of $n$ multinomial distributions (see \eqref{def.product.distribution}). 
The computational virtue of this trick is to replace an intractable sum over $Q^n$ terms (see \eqref{exp.likelihood}) by a sum over only $n^2$ terms (Eq.~\eqref{exp.variatonal.functional}).

Let us define $\mathcal{D}_n$ as a set of product multinomial distributions
\begin{align} \label{def.product.distribution}
    \mathcal{D}_n = \acc{ D_{\taun}= \prod_{i=1}^n\M(1,\tau_{i,1},\ldots,\tau_{i,Q}) \mid \taun \in \mathcal{S}_n} \enspace,
\end{align}
where
\begin{align*}
\mathcal{S}_n=\acc{ \taun = \paren{\tau_1,\ldots,\tau_n} \in \paren{[0,1]^Q}^n \mid \forall i,\ \tau_i=\paren{\tau_{i,1},\ldots,\tau_{i,Q}}, \sum_{q=1}^Q \tau_{i,q}=1} \enspace.
\end{align*}
For any $D_{\taun} \in \mathcal{D}_n$, the variational
log-likelihood, $\J(\cdot; \cdot, \cdot, \cdot)$ is defined by
\begin{align} \label{eq.variational.approx}
     \J(\Xn; \taun, \alpha, \pi) = \L_2(\Xn;  \alpha, \pi) - K\paren{D_{\taun},P^{\Xn}}\enspace,
 \end{align}
where $K(.,.)$ denotes  the Kullback-Leibler divergence, and $P^{\Xn}=\P\paren{\Zn=\cdot\mid \Xn}$.
With this choice of $\mathcal{D}_n$, $\J(\Xn; \taun, \alpha, \pi)$ has the following expression (see \citet{DPR} and the proof of Lemma~\ref{lem.Daudin.et.al}):
\begin{align} \label{exp.variatonal.functional}
\J(\Xn; \taun, \alpha, \pi)= \sum_{i  \neq  j} \sum_{q,l} b_{ij}(q,l) \tau_{i,q} \tau_{j,l} - \sum_{iq} \tau_{i,q} \paren{\log \tau_{i,q} -  \log\alpha_{q}} \enspace,
\end{align}
where $ b_{ij}(q,l)= X_{i,j}\log \pi_{q,l}+ (1-X_{i,j})\log(1-\pi_{q,l})$.
The main improvement of Eq.~\eqref{exp.variatonal.functional} upon Eq.~\eqref{exp.likelihood} is that $\J(\Xn; \taun, \alpha, \pi)$ can be fully computed for every $(\taun, \alpha, \pi)$.  
The variational approximation $R_{\Xn}$ to $P^{\Xn}$ is given by solving the minimization problem over $\mathcal{D}_n$:
\begin{align*}
    R_{\Xn} \in \argmin _{D_{\tau}\in\mathcal{D}_n} K\paren{D_{\tau},P^{\Xn}}\enspace,
 \end{align*}
as long as such a minimizer exists, which amounts to maximizing $\J(\Xn; \taun, \alpha, \pi) $ as follows
\begin{align*}
\tahn = \tahn(\pi,\alpha) := \argmax_{\taun} \J(\Xn; \taun, \alpha, \pi) \enspace.
\end{align*}
The variational estimators (VE) of $(\alpha^*,\pi^*)$ are
\begin{align} \label{def.tau.var}
(\widetilde{\alpha},\widetilde{\pi}) = \argmax_{\alpha,\pi}\J(\Xn; \tahn, \alpha, \pi) \enspace.
\end{align}
Note that in practice, the variational algorithm maximizes $\J(\Xn; \tau, \alpha, \pi)$ alternatively with respect to $\tau$ and $(\alpha,\pi)$  \citep[see][]{DPR}.
Furthermore since $(\alpha,\pi)\mapsto \J(\Xn; \tahn, \alpha, \pi)$ is not concave, this variational algorithm can lead to local optima in the same way as for likelihood optimization.

\medskip

In the sequel, the same notation as in Section~\ref{sec.MLE.SBM} is used.
In particular it is assumed that a realization of SBM is observed, which has been generated from the sequence of true labels $Z=z^*$.
In this setting, $P^*=P^*_{\alpha^*,\pi^*}$ (Section~\ref{subsec.main.asymptotic.result}) denotes the conditional distribution $\P\paren{ \cdot \mid Z=z^*}$ given the whole label sequence.
The first result provides some assurance about the reliability of the variational approximation to $P^{\Xn}=\P_{\alpha^*,\pi^*}\croch{\Zn=\cdot \mid \Xn}$ (Section~\ref{subsec.main.asymptotic.result}).
\begin{prop} \label{prop.kullback.convergence.zero}
For every $n$, let $\mathcal{D}_n$ denote the set defined by \eqref{def.product.distribution}, and $P^{\Xn}\paren{\cdot}$ be the distribution of $\Zn$ given $\Xn$. Then, assuming \eqref{assum.pi.block.diff}--\eqref{assum.alpha.trunc} hold,
\begin{align*}
K(R_{\Xn},P^{\Xn})\defegal \inf_{D \in \mathcal{D}_n} K(D,P^{\Xn}) \xrightarrow[n\to\infty]{}0\enspace,\qquad P^*-a.s. \enspace.
\end{align*}
\end{prop}
Note that this convergence result is given with respect to $P^*$ (and not to $\P$).
Stronger results can be obtained (see Section~\ref{subsec.variational.approx}) thanks to fast convergence rates.
Proposition~\ref{prop.kullback.convergence.zero} yields some confidence in the reliability of the variational approximation, which gets closer to $P^{\Xn}$ as $n$ tends to $+\infty$.
However, it does not provide any warranty about the good behavior of variational estimators, which is precisely the goal of following Section~\ref{subsubsec.uniform.asymptotics.J.L}.

\smallskip

\begin{proof}[Proof of Proposition~\ref{prop.kullback.convergence.zero}]
~\\

By definition of the variational approximation,
\begin{align*}
    K(R_{\Xn},P^{\Xn})  \leq K(\delta_{\zn^*},P^{\Xn}) \enspace,
\end{align*}
where $\delta_{\zn^*}=\prod_{1\leq i\leq n} \delta_{z^*_i} \in \mathcal{D}_n$.
Then,
\begin{align*}
    K(R_{\Xn},P^{\Xn})  \leq K(\delta_{\zn^*},P^{\Xn}) = -\log\croch{P\paren{\Zn=\zn^* \mid \Xn}}
     \enspace,
\end{align*}
since
 \begin{align*}
     K(\delta_{\zn^*},P^{\Xn}) & = \sum_{\zn}\delta_{\zn^*}(\zn) \log\left[ \frac{\delta_{\zn^*}(\zn)}{P^{\Xn}(\zn)}\right] \\
& =-\log\croch{P\paren{\Zn=\zn^* \mid \Xn}}
     \enspace.
\end{align*}
The conclusion results from Theorem~\ref{thm.distrib.conv.zn}, and Corollary~\ref{cor.distrib.conv.as.zn} since $P\paren{\Zn=\zn^* \mid \Xn}\xrightarrow[n\to\infty]{}~1\quad P^*-~a.s.\,$.
\end{proof}

\subsection{Consistency of the variational estimators}\label{subsubsec.uniform.asymptotics.J.L}

Since the variational log-likelihood $\mathcal{J}(\cdot;\cdot,\cdot,\cdot)$ \eqref{eq.variational.approx} is defined from the log-likelihood $\mathcal{L}_2(\cdot;\cdot,\cdot)$ , the properties of $\J(\Xn; \taun, \alpha, \pi)$ are strongly connected to those of $\L_2(\Xn;  \alpha, \pi)$.
Therefore, the strategy followed in the present section is very similar to that of Section~\ref{sec.MLE.SBM}.
In particular, the consistency of $\widetilde{\pi}$ (VE of $\pi^*$, see~\eqref{def.tau.var}) is addressed first.
The consistency of the VE of $\alpha^*$ ($\widetilde{\alpha}$, see~\eqref{def.tau.var}) is handled in a second step and depends on the convergence rate of the estimator of $\widetilde{\pi}$.

\medskip

The first step consists in applying Theorem~\ref{thm.Mestim.general} to settle the $\widetilde{\pi}$ consistency.
Following results aim at justifying the use of Theorem~\ref{thm.Mestim.general} by checking its assumptions.

Theorem~\ref{thm.uniform.P-as.Var.Likelihood} states that $\L_2$
and $\J$ are asymptotically equivalent uniformly with respect to
$\alpha$ and $\pi$.
\begin{thm}\label{thm.uniform.P-as.Var.Likelihood}
With the same notation as Theorem~\ref{thm.unif.conv} and Section~\ref{subsec.variational.approx},
let us define
\begin{align*}
    J_n\paren{\alpha,\pi} \defegal \frac{1}{n(n-1)}\J\paren{\Xn;
\tahn, \alpha, \pi}\enspace.
\end{align*}
Then, \eqref{assum.pi.trunc} and \eqref{assum.alpha.trunc} lead to
\begin{align*}
    \sup_{\alpha,\pi}\acc{ \abs{J_n\paren{\alpha,\pi}- M_n( \alpha, \pi)} } = o\paren{1},\quad \P-a.s.\enspace,
\end{align*}
where the supremum is computed over sets fulfilling \eqref{assum.pi.trunc} and \eqref{assum.alpha.trunc}.
\end{thm}
This statement is stronger than Proposition~\ref{prop.kullback.convergence.zero}
in several respects.
On the one hand, convergence applies almost surely with respect to $\P$ and not $P^*$.
On the other hand, Theorem~\ref{thm.uniform.P-as.Var.Likelihood} exhibits the convergence rate toward 0, which is not faster than $n(n-1)$.

\smallskip

\begin{proof}[Proof of Theorem~\ref{thm.uniform.P-as.Var.Likelihood}]
~\\
From the definitions of $\L_1$, $\L_2$, and $\J$ (respectively given by Eq.~\eqref{def.log-likelihood.L1}, Eq.~\eqref{def.log-likelihood.SBM}, and Eq.~\eqref{eq.variational.approx}) and recalling $\znh =\znh (\pi)=\argmax_{\zn}\L_1(\Xn; \zn, \pi)$,
Lemma~\ref{lem.unif.conv.var} leads to
\begin{align*}
\J(\Xn; \taun, \alpha, \pi) \leq \L_2(\Xn;  \alpha, \pi) \leq  \L_1(\Xn; \znh,\pi) \enspace.
\end{align*}
Then applying \eqref{assum.alpha.trunc} and Lemma~\ref{lem.unif.gap.likelihoods.variation}, there exists $0<\gamma<1$ independent of $(\alpha,\pi)$ such that
\begin{align*}
\abs{\J(\Xn; \tahn, \alpha, \pi)-\L_1(\Xn; \znh, \pi)} \leq &\ n \log (1/\gamma) \enspace.
\end{align*}
The conclusion results straightforwardly.
\end{proof}

\medskip

The consistency of $\widetilde{\pi}$ is provided by the following result, which is a simple consequence of Theorem~\ref{thm.uniform.P-as.Var.Likelihood}, Proposition~\ref{prop.unif.conv.cond.model}, and Theorem~\ref{thm.Mestim.general}.
\begin{cor}
    With the notation of Theorem~\ref{thm.uniform.P-as.Var.Likelihood} and assuming  \eqref{assum.pi.block.diff}, \eqref{assum.pi.trunc}, and \eqref{assum.alpha.trunc} hold, let us define the VE of $(\alpha^*,\pi^*)$
\begin{align*}
(\widetilde{\alpha},\widetilde{\pi}) = \argmax_{\alpha,\pi} J_n(\alpha, \pi) \enspace.
\end{align*}
Then for any distance $d(\cdot,\cdot)$ on the set of $\pi$ parameters,
\begin{align*}
    d( \widetilde{\pi}, \pi^*) \xrightarrow[n \rightarrow +\infty]{\P} 0 \enspace.
\end{align*}
 \end{cor}
The proof is completely similar to that of Corollary~\ref{cor.consist.pi.MLE} and is therefore not reproduced here.

\medskip

Finally, the consistency for the VE of $\alpha^*$ is derived from the same \emph{deconditioning argument} as that one used for the MLE (proof of Theorem~\ref{thm.consist.alpha.MLE}). 
Consistency for $\widetilde{\alpha}$ is stated by the following result where a convergence rate of $1/n$ is assumed for $\widetilde{\pi}$. Note that at least some empirical evidence and heuristics exist \citep[see][]{GDR} in favour of this rate.
\begin{thm}\label{thm.consist.alpha.VE}
Let us assume the VE $\widetilde{\pi}$ converges at rate $1/n$ to $\pi^*$.
With the same assumptions as Theorem~\ref{thm.uniform.P-as.Var.Likelihood} and assuming \eqref{assum.pi.block.diff}, \eqref{assum.pi.trunc}, and \eqref{assum.alpha.trunc} hold, then
\begin{align*}
    d( \widetilde{\alpha}, \alpha^*) \xrightarrow[n \to +\infty]{\P} 0 \enspace,
\end{align*}
where $d$ denotes any distance between vectors in $\R^Q$.
\end{thm}
The crux of the proof is the use of Proposition~\ref{prop.distrib.conv.zn.2}.

\smallskip

\begin{proof}[Proof of Theorem~\ref{thm.consist.alpha.VE}]
~\\
Let us show that given $Z_{[n]}=\zn^*$,
 \begin{align*}
\abs{ \widetilde{\alpha}_q - N_q(\zn^*)/n }\xrightarrow[n\to \infty]{P^*} 0\enspace.
 \end{align*}

For every $q$,
\begin{align*}
    \widetilde{\alpha}_q & = \frac{1}{n}\sum_{i=1}^n \widetilde{\tau}_{i,q},
\end{align*}
where $\widetilde{\tau}_{i,q}=\tah_{i,q}\paren{\widetilde{\alpha}, \widetilde{\pi}}$ (see~\eqref{def.tau.var}).
Introducing $z_i^*$, it comes that
\begin{align*}
    \widetilde{\alpha}_q & = \frac{1}{n}\sum_{i=1}^n \widetilde{\tau}_{i,z_i^*} \1_{(z_i^*=q)} + \frac{1}{n}\sum_{i=1}^n \widetilde{\tau}_{i,q} \1_{(z_i^*\neq q)} \enspace.
\end{align*}

From~\eqref{def.product.distribution}, let us consider the {\it a posteriori } distribution of $\widetilde{Z}_{[n]}=(\widetilde{Z}_1,\ldots,\widetilde{Z}_n)$ denoted by
\begin{align*}
    D_{\widetilde{\tau}_{[n]}}(\zn) = \P\croch{\widetilde{Z}_{[n]}=\zn \mid \Xn} = \prod_{i=1}^n \widetilde{\tau}_{i,z_i} \enspace.
\end{align*}
Then,
\begin{align*}
    \abs{ \widetilde{\alpha}_q - N_q(\zn^*)/n } & = \abs{ \frac{1}{n}\sum_{i=1}^n \paren{ \widetilde{\tau}_{i,z_i^*}-1 } \1_{(z_i^*=q)} + \frac{1}{n}\sum_{i=1}^n \widetilde{\tau}_{i,q} \1_{(z_i^*\neq q)}}\\
  & \leq  \frac{1}{n}\sum_{i=1}^n \paren{ 1- \widetilde{\tau}_{i,z_i^*} } \1_{(z_i^*=q)}  +  \frac{1}{n}\sum_{i=1}^n \widetilde{\tau}_{i,q} \1_{(z_i^*\neq q)}\\
  & \leq  \frac{1}{n}\sum_{i=1}^n \paren{ 1- \widetilde{\tau}_{i,z_i^*} }\enspace ,
\end{align*}
using that when $z_i^*\neq q$, $\widetilde{\tau}_{i,q}\leq \sum_{q\neq z_i^*} \widetilde{\tau}_{i,q} = 1-\widetilde{\tau}_{i,z_i^*} $.
Hence,
\begin{align*}
    \abs{ \widetilde{\alpha}_q - N_q(\zn^*)/n } & \leq  \frac{1}{n}\sum_{i=1}^n  \P\croch{\widetilde{Z}_{[n]} \neq \zn^* \mid \Xn} = 1- D_{\widetilde{\tau}_{[n]}}(\zn^*) \enspace.
\end{align*}

It remains to show $D_{\widetilde{\tau}_{[n]}}(\zn^*)\xrightarrow[n\to \infty]{P^*} 1$ at a rate which does not depend of $\zn^*$.
Let $\widetilde{P}=\P_{\widetilde{\alpha},\widetilde{\pi}}\paren{\Zn=\cdot\mid \Xn}$ denote the {\it a posteriori} distribution of $\Zn$ with parameters $(\widetilde{\alpha},\widetilde{\pi})$ (Section~\ref{subsec.main.asymptotic.result}).
Since Lemma~\ref{lem.aposteriori.new.param} provides
\begin{align*}
    \abs{ D_{\widetilde{\tau}_{[n]}}(\zn^*) - \widetilde{P}(\zn^*) } \leq \sqrt{-\frac{1}{2} \log\croch{ \widetilde{P}(\zn^*) } } \enspace,
\end{align*}
the conclusion results from another use of Proposition~\ref{prop.distrib.conv.zn.2} applied with $\pih = \widetilde{\pi}$ and $v_n=1/n$.

\end{proof}


 \section{Conclusion} \label{sec.conclusion}

This paper provides theoretical (asymptotic) results about the stochastic block model (SBM) inference. Identifiability of SBM parameters has been proved for directed (and undirected) graphs.
This is typically the setting of real applications such as biological networks.

In particular, asymptotic equivalence between maximum-likelihood and variational estimators is proved, as well as the consistency of resulting estimators (up to an additional assumption for the group proportions).
To the best of our knowledge, these are the first results of this type for variational estimators of the SBM parameters.
Such theoretical properties are essential since they validate the empirical practice which uses variational approaches
as a reliable means to deal with up to several thousands of vertices.

Besides, this work can be seen as a preliminary step toward a deeper analysis of maximum-likelihood and variational estimators of SBM parameters.
In particular a further interesting question is the choice of the number $Q$ of classes in the mixture model.
Indeed it is crucial to develop a \textit{data-driven} strategy to choose $Q$ in order to make the variational approach fully applicable in practice and validate the empirical practice.

\newpage

\appendix

\section{Proof of Theorem~\ref{thm.identifiability2}} \label{appendix.identif.2}
\begin{proof}[Proof of Theorem~\ref{thm.identifiability2}]
~\\
Let us just assume $Q=2$, $n=4$, and that no element of $\alpha$ is zero.

If the coordinates of $r=\pi\alpha$ are distinct, then Theorem~\ref{thm.identifiability1} applies and the desired result follows.

Otherwise, the two coordinates are $r$, $r'$ and $r''$ are not distinct.
Set $r_1=r_2=a$ and $u_i=\alpha_1 r_1^i+ \alpha_2 r_2^i$, for $i\geq 0$.
Let us also define $b=r_1'=r_2'$, and $c=r_1''=r_2''$.
Then, the following equalities hold:
\begin{align*}
a & = \pi_{11}\alpha_1+\pi_{12}\alpha_2=\pi_{21}\alpha_1+\pi_{22}\alpha_2   \enspace,\\
b & =\pi_{11}\alpha_1+\pi_{21}\alpha_2=\pi_{12}\alpha_1+\pi_{22}\alpha_2    \enspace,\\
c & =\pi_{11}^2\alpha_1+\pi_{21}\pi_ {12}\alpha_2=\pi_{12}\pi_{21}\alpha_1+\pi_{22}^2\alpha_2\enspace.
\end{align*}

From $a-b=(\pi_{12}-\pi_{21})\alpha_2=-(\pi_{12}-\pi_{21})\alpha_1$ we deduce $\pi_{12}=\pi_{21}$ and $a=b$.
Then,
\begin{eqnarray*}
\alpha_1\alpha_2(\pi_{11}-\pi_{12})^2 &=& (\alpha_1+\alpha_2)(\alpha_1\pi_{11}^2+\alpha_2\pi_{12}^2)-(\alpha_1\pi_{11}+\alpha_2\pi_{12})^2 \\
               &=& c-a^2 \\
               &=& c-b^2 \\
               &=& \alpha_1\alpha_2(\pi_{22}-\pi_{12})^2\enspace.
\end{eqnarray*}

If $c=a^2$, then $\pi_{11}=\pi_{12}=\pi_{21}=\pi_{22}=a$ and $\alpha$ cannot be found.

If $c\neq a^2$, then $|\pi_{11}-\pi_{12}|=|\pi_{22}-\pi_{12}|\neq0$.
But $\alpha_1(\pi_{11}-\pi_{12})=a-\pi_{12}=b-\pi_{12}=\alpha_2(\pi_{22}-\pi_{12})$ leads to $|\alpha_1|=|\alpha_2|$ and $\alpha_1=\alpha_2=1/2$.
Hence $\pi_{11}=\pi_{22}$.
Then, $\pi_{11}$ and $\pi_{12}$ are the roots of the polynomial $x^2-2ax+2a^2-c$.

At this stage, we need to distinguish between $\pi_{11}$ and $\pi_{12}$.
Let us introduce the probability $d$ that $X_{[n]}$ fits the pattern
\begin{center}
\begin{tabular}{cccc}
. & 1 & . & .  \\
. & . & 1 & .  \\
1 & . & . & . \\
. & . & . & .  \\
\end{tabular}\enspace.
\end{center}
Then, $d=(\pi_{11}^3+3\pi_{11}\pi_{12}^2)/4$ and one can compute $e=\root3\of{d-a^3}=(\pi_{11}-\pi_{12})/2$.
This leads to $\pi_{11}=\pi_{22}=a+e$ and $\pi_{12}=\pi_{21}=a-e$, which yields the conclusion.
\end{proof}

\newpage

\section{Proof of Theorem~\ref{thm.distrib.conv.zn}}\label{append.conv.loi.Z}

\subsection{Preliminaries}
Assuming \eqref{assum.pi.block.diff} holds true, $\pi^*$ can be permutation-invariant (see Section~\ref{sec.equiv.class}).
For this reason, we will consider equivalence classes denoted by $\croch{\zn} = \croch{\zn}_{\pi^*}$ for the label vector $\zn$.

Let us define $\PXn(\zn) = \P\croch{\Zn=\zn\mid \Xn}$ for every label vector $\zn$, and $\PXn\paren{ [\zn] } = \P\paren{[\Zn]=[\zn]\mid \Xn}$ for corresponding class $[\zn]$. 
Since every $\zn'\in\croch{\zn}$ satisfies $\PXn(\zn') = \PXn(\zn)$, it results that 
\begin{align}\label{exp.sum.proba.class.vector}
  \PXn\paren{ [\zn] } = \sum_{\zn'\in [\zn]} \PXn(\zn')  = \abs{[\zn]} \PXn(\zn) \enspace,
\end{align}
where $\abs{[\zn]}$ denotes the cardinality of $[\zn]$.

\subsection{Upper bounding $\P \croch{ \sum_{[\zn] \neq
[\zn^*]}\frac{\P\croch{[\Zn]=[\zn]\mid \Xn}}{\P\croch{[\Zn]=[\zn^*]\mid \Xn}} > t  \mid  Z=z^*}$}

Using $P^*$ instead of $\P\croch{\cdot \mid Z=z^*}$ for simplicity, let us first notice 
\begin{align*}
 \sum_{[\zn] \neq [\zn^*]} \frac{\PXn([\zn])}{\PXn([\zn^*])} & =  \sum_{[\zn] \neq [\zn^*]} \frac{ \sum_{\zn'\in[\zn]}\PXn(\zn)}{\sum_{\zn^0\in[\zn^*]}\PXn( \zn^0 )} \\
& \leq  \sum_{[\zn] \neq [\zn^*]} \sum_{\zn'\in[\zn]} \frac{ \PXn(\zn)}{\PXn(\zn^*)} \\
& = \sum_{\zn \not\in [\zn^*]} \frac{ \PXn(\zn)}{\PXn(\zn^*)} \enspace,
\end{align*}
by \eqref{exp.sum.proba.class.vector} applied to $[\zn^*]$.
Partitioning according to the number $\norm{\zn - \zn^*}_0 = r$ of differences between $\zn$ and $\zn^*$, it comes
\begin{align}\label{eq.class.vector}
  \sum_{[\zn] \neq [\zn^*]} \frac{\PXn([\zn])}{\PXn([\zn^*])} & = \sum_{r=1}^n \sum_{ \small
\begin{array}{c}
\zn \not\in [\zn^*] \\
 \norm{\zn - \zn^*}_0 = r
\end{array}} 
\frac{ \PXn(\zn)}{\PXn([\zn^*])} \enspace\cdot
\end{align}
Note that the number of vectors $\zn$ such that $\zn \not\in [\zn^*]$ and $ \norm{\zn - \zn^*}_0 = r$ is roughly upper bounded by 
${n\choose r} \paren{Q-1}^r$,
this upper bound being reached for instance when  $\pi^*$ is such that $\pi^*_{q,l}\neq \pi^*_{q',l'}$ for every $(q,l)\neq (q',l')$.
Then, a straightforward union bound leads to
\begin{align*}
& P^* \croch{ \sum_{[\zn] \neq [\zn^*]} \frac{\PXn([\zn])}{\PXn([\zn^*])} > t } \\ 
 \leq &\ \sum_{r=1}^n \sum_{\small \begin{array}{c}
                        \zn \not\in [\zn^*]\\
				    \norm{\zn-\zn^*}_0=r
                      \end{array}} 
P^* \croch{
\frac{\PXn(\zn)}{\PXn(\zn^*)} > \frac{t}{n^{r+1}(Q-1)^{r}} }
 \enspace,
\end{align*}
by use of ${n\choose r} \leq n^r$, which is tight enough for our purpose.

\subsection{Upper bounding $\P \left[ \frac{\PXn(\zn)}{\PXn(\zn^*)} >
\frac{t}{n^{r+1}(Q-1)^{r}} \ | \ Z=z^* \right]$}

Let us first notice that for every vectors $\zn$ and $\zn^*$,
\begin{align} 
    & \log \paren{\frac{\PXn(\zn)}{\PXn(\zn^*)}} - \E^{Z=z^*}\croch{\log \paren{\frac{\PXn(\zn)}{\PXn(\zn^*)}}} \nonumber \\
= &\ \sum_{i\neq j} \acc{\paren{X_{i,j}-\pi^*_{z^*_i,z^*_j}}\log\croch{\frac{\pi^*_{z_i,z_j}\paren{1-\pi^*_{z^*_i,z^*_j}}}{\pi^*_{z^*_i,z^*_j}\paren{1-\pi^*_{z_i,z_j}}}}}\enspace. \label{eq.sum.indep.rand.var}
\end{align}
Note that for any vector $\zn$ such that $\pi^*_{z_i,z_j} \in \acc{0,1}$ and $\pi^*_{z^*_i,z^*_j}\neq \pi^*_{z_i,z_j}$, $\log \paren{\frac{\PXn(\zn)}{\PXn(\zn^*)}} = -\infty$ and $\PXn(\zn)=0$. Then such vector $\zn$ can be removed from the sum in Eq.~\eqref{eq.class.vector}.

Second for any vectors $\zn$ and $\zn'$, let us further define 
\begin{align*}
  D(\zn,\zn') =   \acc{(i,j)\mid i\neq j,\ \pi^*_{z_i,z_j}\neq \pi^*_{z'_i,z'_j}}  \enspace,
\end{align*}
where $z_i$ and $z'_j$ respectively refer to the $i-$th (resp. $j-$th) coordinate of vector $\zn$ (resp. $\zn'$).
Note that $D(\zn,\zn')$ remains unchanged if $\zn$ and $\zn'$ are replaced by any representatives of their respective classes.

If $N_r(\zn) = \abs{ D(\zn,\zn^*) }$ denotes the number of terms in the sum of Eq.~\eqref{eq.sum.indep.rand.var}, then
\begin{align*}
 & P^* \croch{ \frac{\PXn(\zn)}{\PXn(\zn^*)} >
\frac{t}{n^{r+1}(Q-1)^{r}} } \\
= &\  P^* \croch{ \log \frac{\PXn(\zn)}{\PXn(\zn^*)} > \log \paren{\frac{t}{n^{r+1}(Q-1)^{r}}} } \\
= &\ P^* \left\{ \frac{1}{N_r(\zn)} \paren{ \log \frac{\PXn(\zn)}{\PXn(\zn^*)} - \E^{Z=z^*}
\croch{ \log \frac{\PXn(\zn)}{\PXn(\zn^*)} } } >  \right.\\
& \left. \hspace*{.5cm} \frac{1}{N_r(\zn)} \paren{ \log
\frac{t}{n^{r+1}(Q-1)^{r}} -\E^{Z=z^*}
\croch{ \log \frac{\PXn(\zn)}{\PXn(\zn^*)} }}  \right\}  \enspace.
\end{align*}
Finally, Hoeffding's inequality (Proposition~\ref{prop.Hoeffding.inequality}) applied with $a_{ij}=-b_{ij}=\log \croch{ \paren{1-\zeta}^2 \zeta^{-2}}$ (see Lemma~\ref{lem.Hoeffding.bound} and \eqref{assum.pi.trunc}), and $L=2(b_{i,j}-a_{i,j})$ provides for any $s>0$,
\begin{align*}
    & P^* \croch{\frac{1}{N_r(\zn)} \paren{\log \frac{\PXn(\zn)}{\PXn(\zn^*)} - \E^{Z=z^*}
\croch{ \log \frac{\PXn(\zn)}{\PXn(\zn^*)}}} > s } \\
  & \hspace*{8.5cm} \leq \mathrm{exp}\paren{\frac{- N_r(\zn) s^2}{L^2}}\enspace.
\end{align*}

\subsection{Conclusion}
One then apply this last inequality with a particular choice of $s$:
\begin{align*}
    s = \frac{1}{N_r(\zn)} \paren{ \log
\frac{t}{n^{r+1}(Q-1)^{r}} -\E^{Z=z^*}
\croch{ \log \frac{\PXn(\zn)}{\PXn(\zn^*)} }}\enspace,
\end{align*}
which leads to
\begin{align*}
    s =  \frac{\log t-(r+1)\log(n) -r\log(Q-1)}{N_r(\zn)} -\frac{1}{N_r(\zn)}\E^{Z=z^*}
\croch{ \log \frac{\PXn(\zn)}{\PXn(\zn^*)} }\enspace.
\end{align*}
With Lemma~\ref{lem.Hoeffding.Expectation}, it is not difficult to show that for large enough values of $n$,
\begin{align*}
s^2 \geq \frac{3}{4}\paren{\frac{1}{N_r(\zn)}\E^{Z=z^*}
\croch{ \log \frac{\PXn(\zn)}{\PXn(\zn^*)} }}^2
\geq  \frac{3}{4}\, (c^*)^2
\enspace,
\end{align*}
and that
\begin{align*}
\mathrm{exp}\paren{\frac{-N_r(\zn) s^2}{L^2}} \leq \mathrm{exp}\paren{\frac{-3 N_r(\zn) (c^*)^2}{4 L^2}} \enspace.
\end{align*}
Using Proposition~\ref{A2}, it results that
\begin{align*}
  N_r(\zn) \geq \frac{\gamma^2}{2} n \norm{\zn-\zn^*}_0 = \frac{\gamma^2}{2} n r \enspace,
\end{align*}
and
\begin{align*}
\ P^* \croch{ \sum_{[\zn] \neq [\zn^*]} \frac{\PXn([\zn])}{\PXn([\zn^*])} > t }
 \leq & \sum_{r=1}^n {n \choose r} (Q-1)^r \exp\paren{- \frac{3 (\gamma c^*)^2 }{8 L^2} n r } \\
 = & \sum_{r=1}^n {n \choose r} \croch{(Q-1) u_n}^r
\enspace ,
\end{align*}
where $u_n=\exp\paren{- \frac{3 (\gamma c^*)^2 }{8 L^2} n }$.

Finally for every $t>0$, $ \log(1+x) \leq x$ for every $x\geq 0$ implies
\begin{align*}
\P \croch{ \sum_{[\zn] \neq [\zn^*]} \frac{\PXn([\zn])}{\PXn([\zn^*])} > t \mid Z=z^*} & \leq \paren{1+(Q-1) u_n}^n-1 \\
& \leq  e^{(Q-1)n u_n}-1 \xrightarrow[n\to +\infty]{}0 \enspace,
\end{align*}
since $n u_n \to 0$ as $n\to +\infty$.
Further noticing that the upper bound does not depend on $\zn^*$, the same result holds with $P^*$ replaced by $\P$.

\subsection{Hoeffding's inequality and related lemmas}
\begin{prop}[Hoeffding's inequality] \label{prop.Hoeffding.inequality}
Let $\acc{Y_{i,j}}_{1\leq i\neq j\leq n}$ independent random variables such that for every $i\neq j$, $Y_{i,j}\in [a_{i,j},b_{i,j}]$ almost surely.
Then, for any $t>0$,
\begin{align*}
    \P\croch{\sum_{i\neq j}^n \paren{Y_{i,j}-\E\croch{Y_{i,j}}} > t} \leq \mathrm{exp}\paren{\frac{-t^2}{\sum_{i\neq j}(b_{i,j}-a_{i,j})^2}}\enspace.
\end{align*}
\end{prop}

\medskip

\begin{lem}[Values of $a_{i,j}$ and $b_{i,j}$]\label{lem.Hoeffding.bound}
Assuming \eqref{assum.pi.trunc} holds for $\pi^*$ with $\zeta>0$, it comes for every $1 \leq i\neq j \leq n$,
\begin{align*}
    \abs{X_{i,j} \log\croch{\frac{\pi^*_{z_i,z_j}\paren{1-\pi^*_{z^*_i,z^*_j}}}{\pi^*_{z^*_i,z^*_j}\paren{1-\pi^*_{z_i,z_j}}}}} \leq
2\log \croch{ \paren{ \frac{1-\zeta}{\zeta}}}\enspace.
\end{align*}

\end{lem}

\medskip

\begin{lem}[Bounding the conditional expectation] \label{lem.Hoeffding.Expectation}
Let us assume \eqref{assum.pi.block.diff}, \eqref{assum.pi.trunc}, \eqref{assum.alpha.trunc}, and \eqref{assum.alpha.empirique.trunc} hold true.
Then for every label vectors $\zn$ and $\zn^*$ such that $\norm{\zn-\zn^*}_0=r$ $(1\leq r\leq n)$, there exist positive constants $c^*=c(\pi^*)$ and $C^*=C(\pi^*)$ such that
\begin{align*}
     0<c^* \leq -\frac{1}{N_r(\zn)}\E^{Z=z^*} \croch{ \log \frac{\PXn(\zn)}{\PXn(\zn^*)} } \leq C^*\enspace,
\end{align*}
where $N_r(\zn) = \abs{ \acc{(i,j)\mid i\neq j,\ \pi^*_{z_i,z_j}\neq \pi^*_{z^*_i,z^*_j}} }$.
\end{lem}

\medskip

\begin{proof}[Proof of Lemma~\ref{lem.Hoeffding.Expectation}]
First,
\begin{align*}
& \E^{Z=z^*} \croch{ \log \frac{\PXn(\zn)}{\PXn(\zn^*)} }\\
= &
\E^{Z=z^*}\croch{   \sum_{i\neq j} \acc{X_{i,j}\log\paren{\frac{\pi^*_{z_i,z_j}}{\pi^*_{z^*_i,z^*_j}}}
+(1-X_{i,j})\log\paren{\frac{1-\pi^*_{z_i,z_j}}{1-\pi^*_{z^*_i,z^*_j}}}}+\sum_i\log\frac{\alpha^*_{z_i}}{\alpha^*_{z^*_i}}}\\
 = & \sum_{i\neq j} \E^{Z=z^*}
\croch{X_{i,j}\log\paren{\frac{\pi^*_{z_i,z_j}}{\pi^*_{z^*_i,z^*_j}}}
+(1-X_{i,j})\log\paren{\frac{1-\pi^*_{z_i,z_j}}{1-\pi^*_{z^*_i,z^*_j}}}}+\sum_i\log\frac{\alpha^*_{z_i}}{\alpha^*_{z^*_i}}\\
= & \sum_{i\neq j}
-\croch{\pi^*_{z^*_i,z^*_j}\log\paren{\frac{\pi^*_{z^*_i,z^*_j}}{\pi^*_{z_i,z_j}}}
+(1-\pi^*_{z^*_i,z^*_j})\log\paren{\frac{1-\pi^*_{z^*_i,z^*_j}}{1-\pi^*_{z_i,z_j}}}}
+\sum_i\log\frac{\alpha^*_{z_i}}{\alpha^*_{z^*_i}}\enspace.
\end{align*}
Note that the first sum in the above expression is actually taken over $(i,j)$ such that $\pi^*_{z_i,z_j}\neq \pi^*_{z^*_i,z^*_j}$.

Second, let us introduce
\begin{align*}
C^* \defegal  \max \acc{2 k\paren{ \pi^*_{q,l},\pi^*_{q',l'}} }\qquad \mathrm{and} \qquad 
c^*  \defegal  \min \acc{ k\paren{ \pi^*_{q,l},\pi^*_{q',l'}}/2 } \enspace,
\end{align*}
where {\it maximum} and {\it minimum} are taken over $\acc{\paren{(q,l),(q',l')} \mid \pi^*_{q,l}\neq \pi^*_{q',l'}}$, and $k(x,y) = x\log (x/y) + (1-x) \log\croch{(1-x)/(1-y)}$ for every $x,y\in(0,1)$.
Then for every $(i,j)$ such that $\pi^*_{z_i,z_j}\neq \pi^*_{z^*_i,z^*_j}$,
\begin{align*}
0 < c^* < k\paren{ \pi^*_{z^*_i,z^*_j}, \pi^*_{z_i,z_j} } < C^*\enspace.
\end{align*}

Third, \eqref{assum.alpha.trunc} implies that
$\abs{\log\frac{\alpha^*_{z_i}}{\alpha^*_{z^*_i}}} \le \log
\frac{1-\gamma}{\gamma}.$
Therefore \eqref{assum.alpha.empirique.trunc} and Proposition~\ref{A2} entail 
\begin{align*}
\frac{1}{N_r(\zn)}
\sum_i\log\frac{\alpha^*_{z_i}}{\alpha^*_{z^*_i}} \leq
\frac{r}{N_r(\zn)} \log \frac{1-\gamma}{\gamma} \leq
\frac{1}{n}
\frac{2}{\gamma^2} \log \frac{1-\gamma}{\gamma} \xrightarrow[n\to
+\infty]{}0 \enspace.  
\end{align*}

The conclusion follows for every $n$ such that
\begin{align*}
  \frac{1}{n}
\frac{2}{\gamma^2} \log \frac{1-\gamma}{\gamma} < \frac{ c^*}{2} < C^* \enspace\cdot 
\end{align*}

\end{proof}

\begin{prop} \label{A2}
Let $\zn$ and $\zn^*$ denote two label vectors.
If \eqref{assum.pi.block.diff} and \eqref{assum.alpha.empirique.trunc} hold true, then
\begin{align*}
\abs{ D(\zn,\zn^*) } \geq \frac{\gamma^2}{2}n \norm{ \zn - \zn^*}_0 \enspace,                                  
\end{align*}
where $ D(\zn,\zn') =  \acc{(i,j)\mid i\neq j,\ \pi^*_{z_i,z_j}\neq \pi^*_{z'_i,z'_j}}$, $\gamma>0$ is the constant given by \eqref{assum.alpha.empirique.trunc}, and $\norm{ \zn - \zn'}_0=\sum_{i=1}^n \1_{(z_i\neq z'_i)}$.
\end{prop}

\medskip

\begin{proof}[Proof of Proposition~\ref{A2}]
 ~\\
Since one assumes \eqref{assum.pi.block.diff} holds true, $\pi$ can be \emph{permutation-invariant} (see Example~\ref{ex.permutation.invariant}).
Then, let us define $\pi^{\sigma}=(\pi_{\sigma(q),\sigma(l)})_{1\leq q,l \leq Q}$ with $\sigma$ a permutation on $\acc{1,\ldots,Q}$.
Note that for permutation-invariant matrix $\pi$, 
there exists a permutation $\sigma\neq Id$ on $\acc{1,\ldots,Q}$
such that $\pi^{\sigma}=\pi$.
Then, the following equalities hold
\begin{align*}
D(\zn,\zn') = D\paren{\sigma(\zn),\zn' } = D\paren{\zn,\sigma(\zn') } \enspace,
\end{align*}
with $\sigma (\zn)=\paren{ \sigma(z_1), \sigma(z_2),\ldots,\sigma(z_n) }$. 
Furthermore, neither $\abs{D(\zn,\zn^*)}$ nor $\norm{\zn-\zn^*}_0$ will change if the same permutation is applied to the coordinates of vectors $\zn$ and $\zn^*$.
Then, computing $\abs{D(\zn,\zn^*)}$ can be made by reordering $\zn$ and $\zn^*$.

Assumption \eqref{assum.alpha.empirique.trunc} implies that the number of coordinates of $\zn^*$ that are equal to 1 is at least $n_{\gamma}\defegal \lceil n\gamma \rceil$, where $\lceil n\gamma\rceil$ denotes the first integer larger than $n\gamma$.
The same property holds for every $1\leq q \leq Q$.
Let us use a permutation of the coordinates of $\zn^*$ such that 
\begin{align*}
\zn^*=(1,2,\ldots,Q,1,2,\ldots,Q,\ldots,1,2,\ldots,Q,z^*_{Qn_{\gamma}+1},z^*_{Qn_{\gamma}+2},\ldots,z^*_n) \enspace,
\end{align*}
and apply the same permutation to $\zn$.
For each block $k$ of $Q$ coordinates $\paren{1,\ldots,Q}$ of $\zn^*$, let us introduce a mapping $\sigma_k(\cdot)$ where $k$ denotes the number of the block in $\zn^*$ such that
\begin{align*}
  \forall k,\quad   Q+1\leq i \leq (k+1)Q,\qquad \sigma_k(z^*_i)=z_i \enspace.
\end{align*}
Then it comes 
\begin{align} \label{eq.def.sigma.zn}
  \zn&=(\sigma_1(1),\sigma_1(2),\ldots,\sigma_1(Q),\sigma_2(1),\sigma_2(2),\ldots,\sigma_2(Q),\ldots,
\sigma_{n_{\gamma}}(1),\ldots,\sigma_{n_{\gamma}}(Q),\\
&\quad z_{Qn_{\gamma}+1},z_{Qn_{\gamma}+2},\ldots,z_n) \nonumber \enspace.
\end{align}

Note that this reorganization of $\zn^*$ is not unique. For instance, it is possible to exchange $\sigma_1(3)$ with $\sigma_4(3)$.
Each $\sigma_k$ is a function from $\acc{1,\ldots,Q}$ to
$\acc{1,\ldots,Q}$, which is a permutation provided it is injective.
Let us choose a reorganization of the coordinates of $z^*$ which
minimizes the number of injective $\sigma_k$s.

Besides, 
\begin{align*}
\abs{  D\paren{\zn,\zn^*}}
& \geq \abs{ \acc{(i,j)\mid i\neq j,\ i,j\leq Qn_{\gamma}\ \pi^*_{z^*_i,z^*_j}\neq \pi^*_{z_i,z_j}}  } \\
& = \sum_{k,k'}\abs{ \acc{(i,j)\mid i\neq j,\ i\in I_k,j\in I_{k'},\ \pi^*_{z^*_i,z^*_j}\neq \pi^*_{z_i,z_j}}  } \enspace,
\end{align*}
where $I_k$ denotes the $k-$th block of coordinates of $\zn^*$.
If $k\neq k'$, the requirement that $i\neq j$ is fulfilled.
Otherwise for $k=k'$, it is necessary to require that $z_i^*\neq z_j^*$ since every values in $I_k$ are different.
Let us denote by $B(k,k')=\abs{ \acc{(q,l) \mid \pi^*_{q,l}\neq \pi^*_{\sigma_k(q), \sigma_{k'}(l) } } }$ and by $B(k)=\abs{ \acc{(q,l) \mid q\neq l,\ \pi^*_{q,l}\neq \pi^*_{\sigma_k(q), \sigma_{k}(l) } } }$.
Then, it comes that
\begin{align*}
\abs{  D\paren{\zn,\zn^*} }
& \geq  \sum_{k,k'} \abs{ \acc{(i,j)\mid i\neq j,\ i\in I_k,j\in I_{k'},\ \pi^*_{z^*_i,z^*_j}\neq \pi^*_{z_i,z_j}}  } \\
 & = \sum_{k\neq k'} \abs{ \acc{(q,l)\mid \  \pi^*_{q,l}\neq \pi^*_{\sigma_k(q), \sigma_{k'}(l) }}  } \\ 
 & + \sum_{k} \abs{ \acc{(q,l)\mid q\neq l,\  \pi^*_{q,l}\neq \pi^*_{\sigma_k(q), \sigma_{k}(l) }}  } \\
 & = \sum_{k\neq k'} B(k,k') + \sum_{k} B(k)
\enspace.
\end{align*}
Therefore, lower bounding $ \abs{ D\paren{\zn,\zn^*} }$ amounts to assess the cardinality of $B(k,k')$ and $B(k)$, for $ 1\leq k\neq k' \leq n_{\gamma}$.

Let us now distinguish between two cases:
\begin{enumerate}
  \item either for every $k,k' \in \acc{1,\ldots,n_{\gamma}}$, $B(k,k')+B(k',k)>0$ and $B(k)>0$.

\item or there exist $k,k'$ such that $B(k,k')+B(k',k)=0$ or $B(k)=0$. 
\end{enumerate}

\paragraph{First case:}
In this setting, let $\norm{\zn-\zn^*}_0=r$. Then,
\begin{align*}
\abs{ D(\zn,\zn^*) } & \geq \sum_{k\neq k'} B(k,k') + \sum_{k} B(k) \\
& = \sum_{k<k'} \croch{B(k,k') + B(k',k)} +  \sum_{k} B(k) \\
& \geq \frac{n_{\gamma}(n_{\gamma}-1)}{2} +  n_{\gamma} = \frac{n_{\gamma}(n_{\gamma}+1)}{2} \\
& \geq \frac{n_{\gamma}^2}{2} \geq \frac{n^2 \gamma^2}{2} 
\geq \frac{\gamma^2}{2} n \,r \enspace,
\end{align*}
since $ n_{\gamma} \geq n \gamma $ and $n\geq r$.

\paragraph{Second case:}
Let us assume that there exist $k,k'$ such that $B(k,k')+B(k',k)=0$. (The $B(k)$s will be lower bounded by 0.) 

Then for every such $k,k'$, $\sigma_k$ and $\sigma_{k'}$ are permutations.
Indeed such $k,k'$ lead to
$\pi_{q,l}=\pi_{\sigma_k(q),\sigma_{k'}(l)}=
\pi_{\sigma_{k'}(q),\sigma_k(l)}$, for every $q,l\in \acc{1,\ldots,Q}$.
Assume furthermore that $\sigma_k(q)=\sigma_k(q')$ for some
$q,q' \in \acc{1,\ldots,Q}$.
Then for every $l \in \acc{1,\ldots,Q}$, $\pi_{q,l}=\pi_{\sigma_k(q),\sigma_{k'}(l)}=\pi_{\sigma_k(q'),\sigma_{k'}(l)}=\pi_{q',l}$.
Hence, for every $l\in \acc{1,\ldots,Q}$, $\pi_{q,l}=\pi_{q',l}$, which implies $q=q'$ using \eqref{assum.pi.block.diff}.
Therefore, $\sigma_k$ is injective and thus a permutation of $\acc{1,\ldots,Q}$. 
The same property holds for $\sigma_{k'}$ which is also a permutation of $ \acc{1,\ldots,Q}$.

Furthermore for any such $k,k'$, $\sigma_k=\sigma_{k'}=\sigma$ and $\pi^{\sigma_k}=\pi$, where $\sigma$ denotes a permutation of $\acc{1,\ldots,Q}$.
Indeed if one assumes $\sigma_k \ne \sigma_{k'}$,
then there exists $q\in\acc{1,\ldots,Q}$ such that
$\sigma_k(q)\neq\sigma_j(q)$. 
If it holds, one can interchange coordinates of $\zn$: $\sigma_k(q)$ and $\sigma_{k'}(q)$. 
This results in new mappings $\sigma_k$ and $\sigma_{k'}$ between $\zn^*$ and $\zn$, which are no longer injective.
Then, the number of injective mappings $\sigma_k$ in the writing of $\zn$ decreases by 2 and is no longer minimal as earlier assumed.
This yields $\sigma_k=\sigma_{k'}$ and thus $\pi^{\sigma_k}=\pi$.
Note that the existence of such a unique permutation $\sigma$ results from the fact that
for every $i>Qn_{\gamma}$, $z_i=\sigma_k(z^*_i)$.
Indeed if this was not true, the same reasoning as before applies:
An interchange between $z_i$ and
$\sigma_k(z^*_i)$ would decrease the number of injective
$\sigma_{k}$s in \eqref{eq.def.sigma.zn}, which contradicts our assumption.
As consequences, it also comes that $\pi^{\sigma}=\pi$ and that for every $i>Qn_{\gamma}$,
$z_i=\sigma(z^*_i)$. 

Let $m$ denote the number of non-injective mappings $\sigma_k$. 
Note that for any non-injective mapping $\sigma_k$ ($1\leq k \leq n_{\gamma})$, there exists at least one difference between $\zn$ and $\zn^*$ in the corresponding block $k$.
Then, the number $r$ of differences satisfies 
\begin{align} \label{ineq.difference.injective}
r\leq m \times Q \qquad \Leftrightarrow \qquad \frac{m}{r} \geq \frac{1}{Q} \enspace.                                               
\end{align}

The conclusion results from 
\begin{align*}
\abs{ D(\zn,\zn^*) } & \geq \sum_{k\neq k'} B(k,k') + \sum_{k} B(k) 
\geq \sum_{k\neq k'} B(k,k') \\
& \geq \frac{ n_{\gamma}(n_{\gamma}-1) -  (n_{\gamma}-m) \croch{n_{\gamma}-m - 1 } }{2} \\
& = \frac{   2m n_{\gamma} - m^2 -m }{2} =  \frac{   m n_{\gamma} +  m\croch{ n_{\gamma}- m - 1 }}{2}
\enspace.
\end{align*}

Finally, let us notice that $m\leq n_{\gamma}$, and that $n_{\gamma}-1\geq m = n_{\gamma}$ amounts to say that no injective mapping $\sigma_k$ exists in \eqref{eq.def.sigma.zn}.
However with the same reasoning as before, it means that for every $1\leq k,k' \leq n_{\gamma}$, $B(k,k')+B(k',k)>0$, which contradicts the assumption.
Then, $n_{\gamma}-(m+1)\geq 0$ and 
\begin{align*}
  D(\zn,\zn^*) & \geq  \frac{   n_{\gamma} m }{2} \geq \frac{   \gamma n m} {2} = \frac{\gamma n\, r   m } {2 r}  \geq \frac{\gamma n\, r   } {2 Q}  \geq \frac{\gamma^2 n\, r   } {2 } \enspace.
\end{align*}
by use of \eqref{ineq.difference.injective} and $\gamma \leq 1/Q$ (see Assumption \eqref{assum.alpha.empirique.trunc}).

\end{proof}

\newpage

\section{Proof of Proposition \ref{prop.unif.conv.cond.model}}\label{Appendix.prop.unif.conv.M1}

\begin{proof}[Proof of Proposition \ref{prop.unif.conv.cond.model}]

Let us first recall that
\begin{align*}
    \phi_n\paren{\zn,\pi} & \defegal \frac{1}{n(n-1)} \L_1\paren{\Xn;\zn,\pi}\enspace, \\
    \Phi_n\paren{\zn,\pi} & \defegal \E\croch{\phi_n\paren{\zn,\pi} \mid \zn=\zn^*}\enspace.
\end{align*}
Then,
\begin{align*}
\abs{ \phi_n\paren{\zn,\pi} - \Phi_n \paren{\zn,\pi} } & = \rho_n \abs{\sum_{i\neq j} \paren{ X_{i,j} - \pi^*_{z^*_i,z^*_j} } \log \croch{ \pi_{z_i,z_j} / (1-\pi_{z_i,z_j}) } } \enspace,\\
    & = \rho_n\abs{\sum_{i\neq j} \xi_{ij}\, g\paren{\pi_{z_i,z_j}}} \enspace,
\end{align*}
\sloppy
where  $\rho_n=\croch{n(n-1)}^{-1}$, $\xi_{ij}=X_{i,j} - \pi^*_{z^*_i,z^*_j}$, and $ g(t)=\log(t/(1-t))$, $t\in]0,1[$.
With $g_{i,j}=g\paren{\pi_{z_i,z_j}}$,
let us introduce
\begin{align*}
S_n\paren{g}=\abs{ \sum_{i\neq j} \xi_{ij}\, g_{ij} }\enspace,
\end{align*}
where $g=\acc{g_{i,j}}_{1\leq i\neq j \leq n}$.
Note that on the parameter set $\mathcal{P}$ defined by  \eqref{eq.admissible.set.parameters}, $\abs{\xi_{ij}\, g_{ij}}<+\infty\enspace a.s.$ for every $1\leq i\neq j \leq n$.

The expected control will result from the use of Talagrand's inequality (Theorem~\ref{thm.Talagrand}).
For every $\zn$ and $\epsilon>0$, let us introduce the set 
\begin{align*}
\mathcal{P}(\zn)=\acc{\pi \mid (\zn,\pi) \in \mathcal{P}} \enspace,
\end{align*}
and define the event
\begin{align*}
\Omega_n(\epsilon; \zn) = \acc{ \sup_{\mathcal{P}(\zn)} \rho_n S_n\paren{g} \leq (1+\epsilon) \sqrt{\rho_n} \Lambda + \sqrt{ \rho_n \Gamma^2 x_n}+ \paren{1/\epsilon+1/3} \rho_n \Gamma\, x_n }\enspace,
\end{align*}
where $\Gamma$ and $\Lambda$ are constants respectively defined in Lemmas~\ref{lem.bounds.Talagrand} and~\ref{lem.expect.upper.bound}, and $\acc{x_n}_n$ is a sequence of positive real numbers to be chosen later.
%
Then Theorem~\ref{thm.Talagrand} implies for any $\zn$
\begin{align*}
    P^*\croch{\Omega_n(\epsilon;\zn)^c}\leq e^{-x_n}\enspace.
\end{align*}
\begin{align*}
 & P^*\croch{ \sup_{\mathcal{P}}\abs{\phi_n\paren{\zn,\pi}-\Phi_n\paren{\zn,\pi}} > \eta } \\
\leq\ & \sum_{\zn} P^*\croch{ \acc{\sup_{\mathcal{P}(\zn)} \rho_n S_n(g)  > \eta} \cap \Omega_n(\epsilon;\zn) } + \sum_{\zn} e^{-x_n} \\
\leq\ & \sum_{\zn} P^*\croch{ (1+\epsilon) \sqrt{\rho_n} \Lambda + \sqrt{ \rho_n \Gamma^2 x_n}+ \paren{1/\epsilon+1/3} \rho_n \Gamma\, x_n > \eta } + \sum_{\zn} e^{-x_n}
\enspace.
\end{align*}
Since $\zn$ belongs to a set of cardinality at most $Q^n$, choosing $x_n = n \log(n)$ entails the first sum is equal to 0 for large enough values of $n$, while the second sum converges to 0.

Finally, a quick inspection of the proof shows this convergence is uniform with respect to $\zn^*$, which provides the expected result.
\end{proof}

\bigskip

\begin{thm}[Talagrand]\label{thm.Talagrand}

Let $\acc{Y_{ij}}_{1\leq i\neq j\leq n}$ denote independent centered random variables, and define
\begin{align*}
    \forall g \in \mathcal{G},\quad S_n(g) = \sum_{i\neq j} Y_{ij} g_{ij}\enspace,
\end{align*}
where $\mathcal{G}\subset \R^{n^2}$.
Let us further assume that there exist $b>0$ and $\sigma^2>0$ such that  $\abs{Y_{ij}g_{ij}}\leq b$ for every $(i,j)$, and  $\sup_{g\in\mathcal{G}}\sum_{i\neq j} \Var(Y_{ij}g_{ij})\leq \sigma^2$.
Then, for every $\epsilon>0$, and $x>0$,
\begin{align*}
 \P\croch{\sup_{g} S_n(g)\geq \E\croch{\sup_{g} S_n(g)}(1+\epsilon)+\sqrt{2\sigma^2 x}+ b\paren{1/\epsilon+1/3}x}\leq e^{-x}\enspace.
\end{align*}

\end{thm}

\medskip
\begin{proof}
A proof can be found in \citet{Mass_2007} (p.170, Eq. (5.50)).

\end{proof}

\bigskip

\begin{lem}\label{lem.bounds.Talagrand}
With the same notation as Theorem~\ref{thm.Talagrand}, Assumption \eqref{assum.pi.trunc} entails that there exists $\Gamma(\zeta)>0$ only depending on $\zeta$ such that
    \begin{align*}
        \sup_{\mathcal{P}}\max_{i\neq j} \abs{\xi_{ij}\, g_{ij}} \leq \Gamma, \quad \mathrm{and}\quad
\sup_{\mathcal{P}} \max_{i\neq j} \Var\paren{\xi_{ij}\, g_{ij}} \leq \frac{\Gamma^2}{4}\enspace.
    \end{align*}
\end{lem}

\medskip

\begin{proof}
If $\paren{\zn,\pi} \in \mathcal{P}$, then
\begin{align*}
\paren{ \pi_{z_i,z_j}\in \acc{0,1} \Rightarrow \pi_{z^*_i,z^*_j}^*= \pi_{z_i,z_j} }
\Rightarrow \paren{ g_{i,j}=0 } \enspace.
\end{align*}
Then for every $\paren{\zn,\pi} \in \mathcal{P}$, there exists $\Gamma=\Gamma(\zeta)>0$ (Assumption \eqref{assum.pi.trunc})  such that
\begin{align*}
\forall i \neq j, \quad \abs{ \xi_{ij}\, g_{ij}} \leq \Gamma\enspace,
\end{align*}
for every $(\zn , \pi)\in \mathcal{P}$.
This also leads to
\begin{align*}
\forall i \neq j, \quad \Var\paren{ \xi_{ij}\, g_{ij}} \leq \Gamma^2/4 \enspace.
\end{align*}
\end{proof}

\bigskip

\begin{lem}\label{lem.expect.upper.bound} With the same notation as Proposition~\ref{prop.unif.conv.cond.model},
for every $\zn$ such that $ (\zn,\pi) \in \mathcal{P}$, there exists a constant $\Lambda=\Lambda(\zeta)>0$ (Assumption \eqref{assum.pi.trunc}) such that
\begin{align*}
\E\croch{\sup_{\mathcal{P}(\zn)}\rho_n\abs{\sum_{i\neq j} \paren{ X_{i,j} - \pi^*_{z^*_i,z^*j} }g_{ij}} \mid Z=z^*}
& \leq \Lambda \croch{n(n-1)}^{\,-1/2}\enspace.
\end{align*}
\end{lem}

\medskip

\begin{proof}[Proof of Lemma~\ref{lem.expect.upper.bound}]
Let $\E^*\croch{\cdot}$ denote the expectation given $Z=z^*$.
Then,
\begin{align*}
& \E^* \sup_{\mathcal{P}(\zn)}\rho_n\abs{\sum_{i\neq j} \paren{ X_{i,j} - \pi^*_{z^*_i,z^*_j} }g_{ij}} \\
\leq &
\ \E^*_{X,X'}\croch{\sup_{\mathcal{P}(\zn)} \rho_n \abs{\sum_{i\neq j}\paren{X_{i,j}-X'_{ij}}g_{ij} } }\enspace,
\end{align*}
where the $X'_{i,j}$s are independent random variables with the same distribution as the $X_{i,j}$s.
A symmetrization argument based on Rademacher variables $\acc{ \epsilon_{i,j} }_{1\leq i \neq j \leq n}$ leads to
\begin{align*}
    & \E^* \sup_{\mathcal{P}(\zn)}\rho_n\abs{\sum_{i\neq j} \paren{ X_{i,j} - \pi^*_{z^*_i,z^*_j} }g_{ij}}\\
\leq \ &   2\E^*\croch{\sup_{\mathcal{P}(\zn)} \rho_n \E_{\epsilon}\croch{\abs{\sum_{i\neq j}\epsilon_{ij}X_{i,j}g_{i,j}}} }\enspace,
\end{align*}
where $\E_{\epsilon}[\cdot]$ denotes the expectation with respect to $\epsilon_{i,j}$s.
Then, Jensen's inequality yields
\begin{align*}
    & \E^* \sup_{\mathcal{P}(\zn)}\rho_n\abs{\sum_{i\neq j} \paren{ X_{i,j} - \pi^*_{z^*_i,z^*_j} }g_{ij}}\\
\leq &\, 2\E^*\croch{\sup_{\mathcal{P}(\zn)} \rho_n \sqrt{\mathrm{Var}_{\epsilon}\croch{\sum_{i\neq j}\epsilon_{ij}X_{i,j} g_{ij} }}}\\
\leq & \,
2 \E^*\croch{\sup_{\mathcal{P}(\zn)} \rho_n \sqrt{n(n-1) g^2_{ij} }} \leq  \Lambda(\zeta) \sqrt{\rho_n}\enspace.
\end{align*}

\end{proof}

\newpage

\section{Theorem~\ref{thm.unif.conv}}
\label{appendix.theorem.consist.MLE}

\subsection{Proof of Theorem~\ref{thm.unif.conv}}

\subsubsection{Notation}

For any metric space $(\Theta,d)$ and any real-valued function $f:\ \Theta \to \R$, let us define $\norm{\cdot}_{\Theta}$ by 
\begin{align*}
  \norm{f}_{\Theta} \defegal \sup_{\theta \in\Theta} \abs{ f(\theta) } \enspace.
\end{align*}
Let also $\alpha^*$ and $\pi^*$ be the true values of $\alpha$ and
$\pi$ in SBM (see Section~\ref{subsec.modelSBM}), $\mathcal{A}$ be the set of  stochastic matrices of size $Q$ given by
$\mathcal{A}=\{A=\paren{a_{k, l}}_{1\leq k,l \leq Q} \mid a_{k,l}\geq 0, \ \sum_{l=1}^Q a_{k,l}=1 \}$.

Furthermore, let us introduce the following quantities
\begin{align*}
\phi_n(\pi,\zn) & =  \frac{1}{n(n-1)}\L_1(\Xn;\zn,\pi),\quad
\znh(\pi)   =  \argmax_z \phi_n(\zn,\pi)\enspace,\\
\Phi_n(\pi,\zn) & =  \frac{1}{n(n-1)}\sum_{i\neq j } \pi^*_{z_i^*z_j^*}\log\pi_{z_i,z_j}+(1-\pi^*_{z_i^*z_j^*})\log(1-\pi_{z_i,z_j}) \enspace,\\
\znt(\pi) & =  \argmax_z \Phi_n(\zn,\pi)\enspace,\\
M_n(\alpha,\pi) & =  \frac{1}{n(n-1)}\L_2(\Xn;\alpha,\pi)\enspace,\\
\mathbb{M}(\pi,A) &  =  \sum_{q,l}\alpha^*_q\alpha^*_l \sum_{q'l'}a_{q,q'}a_{l,l'}[\pi^*_{q,l}\log\pi_{q'l'}+
(1-\pi^*_{q,l})\log(1-\pi_{q'l'})]\enspace,\\
 \bar{A}_{\pi} & = \argmax_{A \in \mathcal{A}} \mathbb{M}(\pi,A),\qquad \mathbb{M}(\pi) = \mathbb{M}(\pi,\bar{A}_{\pi}) \enspace.
\end{align*}
Note that $\bar{A}_{\pi}$ is not necessarily unique for every $\pi$. However our analysis only requires unicity of $\bar{A}_{\pi^*}$ and $\bar{A}_{\pi^*}=I_Q$, which is proved in the following reasoning.
Furthermore $\mathbb{M}(\pi^*)=\sum_{q,l}\alpha^*_q\alpha^*_l \mathbb{H^*}_{q,l}$, where $\mathbb{H^*}_{q,l}=\pi^*_{q,l}\log\pi^*_{q,l}+(1-\pi^*_{q,l})\log(1-\pi^*_{q,l})$.

\subsubsection{Proof}

First let us prove $\bar{A}_{\pi^*}$ is unique and $\bar{A}_{\pi^*} = I_Q$. 
Let us assume $\bar{A}_{\pi^*} \neq I_Q$. By definition of $\bar A_{\pi}$, it results
\begin{align*}
0 \leq  &  \mathbb{M}(\pi^*,\bar{A}_{\pi^*})-\mathbb{M}(\pi^*,I_Q)\\
= &   \sum_{q,l}\alpha^*_q\alpha^*_l\sum_{q'l'}\bar{a}_{q,q'}(\pi^*)\bar{a}_{l,l'}(\pi^*)[\pi^*_{q,l}\log\frac{\pi^*_{q'l'}}{\pi^*_{q,l}}+
(1-\pi^*_{q,l})\log\frac{1-\pi^*_{q'l'}}{1-\pi^*_{q,l}}]\enspace\\
= &  - \sum_{q,l}\alpha^*_q\alpha^*_l\sum_{q'l'}\bar{a}_{q,q'}(\pi^*)\bar{a}_{l,l'}(\pi^*) K(\pi^*_{q,l},\pi^*_{q'l'}) \leq 0 \enspace.
\end{align*}
Therefore for every $(q,q',l,l')$, $\bar{a}_{q,q'}(\pi^*)\bar{a}_{l,l'}(\pi^*)K(\pi^*_{q,l},\pi^*_{q'l'})=0$ by \eqref{assum.alpha.trunc}.

Since $\sum_{l'}\bar{a}_{l,l'}(\pi^*)=1$ implies for every $ 1\leq l\leq Q$, there exists $1\leq l'\leq Q$ such that $\overline{a}_{l,l'}(\pi^*) >0$, there exists $f: \acc{1,\ldots,Q}
\to \acc{1,\ldots,Q}$ such that $ \pi^*_{q,l} = \pi^*_{f(q),f(l)} $.
Then
\begin{itemize}
  \item $f$ is a permutation of $\acc{1,\ldots,Q}$ is excluded since we are working up to label switching,

  \item otherwise there exist two indices $q_1$ and $q_2$  ($q_1\neq q_2$) such that rows $q_1$ and $q_2$ of $\pi^*$ are equal and so do the corresponding columns, which is excluded by \eqref{assum.pi.block.diff},
\end{itemize}
which proves the unicity and that $\bar A_{\pi^*}=I_Q$.

\medskip

Second, let us prove that: $\forall \eta>0,\quad \sup_{d(\pi,\pi^*) \ge \eta } \mathbb{M}(\pi)<\mathbb{M}(\pi^*)$.
In the sequel, let $\paren{\bar{a}_{q,l}}_{1\leq q,l \leq Q}$ denote coefficients of $\bar{A}_{\pi}$.
Without further indication, $\bar{a}_{q,l}$ depends on the matrix $\pi$.
Then,
\begin{align*}
\quad & \mathbb{M}(\pi) - \mathbb{M}(\pi^*) \\
= &   \sum_{q,l}\alpha^*_q\alpha^*_l\sum_{q'l'}\bar{a}_{q,q'}\bar{a}_{l,l'}[\pi^*_{q,l}\log\frac{\pi_{q'l'}}{\pi^*_{q,l}}+
(1-\pi^*_{q,l})\log\frac{1-\pi_{q'l'}}{1-\pi^*_{q,l}}]\enspace\\
= &  - \sum_{q,l}\alpha^*_q\alpha^*_l\sum_{q'l'}\bar{a}_{q,q'}\bar{a}_{l,l'}K(\pi^*_{q,l},\pi_{q'l'})\enspace.
\end{align*}
Since $\acc{\pi \mid d(\pi,\pi^*) \geq \eta,\ \eqref{assum.pi.block.diff},\ \eqref{assum.pi.trunc}}$ is a compact set, there exists $\pi^0 \neq \pi^*$ satisfying \eqref{assum.pi.block.diff} and \eqref{assum.pi.trunc} such that
\begin{align*}
\sup_{d(\pi,\pi^*) \ge \eta } \mathbb{M}(\pi) - \mathbb{M}(\pi^*)
= \mathbb{M}(\pi^0) - \mathbb{M}(\pi^*) <0 \enspace.
\end{align*}
Otherwise for every $(q,l)$, $\sum_{q'l'}\bar{a}_{q,q'}\bar{a}_{l,l'}K(\pi^*_{q,l},\pi^0_{q'l'})=0$ would entail by \eqref{assum.alpha.trunc} that for every $(q,l,q',l')$,
$\bar{a}_{q,q'}\bar{a}_{l,l'}K(\pi^*_{q,l},\pi^0_{q'l'})=0$.
It implies that there exists $f: \acc{1,\ldots,Q}
\to \acc{1,\ldots,Q}$ such that $ \pi^*_{q,l} = \pi^0_{f(q),f(l)} $.
The same reasoning as for the unicity of $\bar A_{\pi^*}$ leads to 
\begin{itemize}
  \item if $f$ is a permutation of $\acc{1,\ldots,Q}$: a contradiction arises since $\pi^0 \neq \pi^*$ up to label switching,

  \item otherwise there exist two indices $q_1$ and $q_2$  ($q_1\neq q_2$) such that rows $q_1$ and $q_2$ of $\pi$ are equal and so do the corresponding columns, which is excluded by \eqref{assum.pi.block.diff}.
\end{itemize}

\medskip

Third, let us prove that $\sup_{\alpha,\pi} \abs{M_n(\alpha,\pi)-\mathbb{M}(\pi)} \xrightarrow[n\to +\infty]{\P}0$.
Set
\begin{eqnarray}
|M_n(\alpha,\pi) - \mathbb{M}(\pi)| &
\leq & |M_n(\alpha,\pi) - \phi_n(\pi,\znh)| \label{ineq.term1}\\
   &+& |\phi_n(\pi,\znh) - \Phi_n(\pi,\znt)| \label{ineq.term2} \\
   &+& |\Phi_n(\pi,\znt) - \mathbb{M}(\pi)| \label{ineq.term3}\enspace.
\end{eqnarray}
These three terms are successively controlled in the following.

\paragraph{Upper bound of \eqref{ineq.term1}:}

Lemma \ref{lem.unif.gap.likelihoods.variation} implies that
$\P-a.s.$,
\begin{align*}
\sup_{\alpha,\pi}\abs{M_n(\alpha,\pi)-\phi_n(\pi,\znh)}
& = \sup_{\alpha,\pi}\frac{\abs{\L_2(\Xn; \alpha, \pi)-\L_1(\Xn;
\pi,\znh)}}{n(n-1)}\\
& \leq\ \frac{\log(1/\gamma)}{n-1}
\xrightarrow[n\to\infty]{} 0 \enspace,
\end{align*}

\paragraph{Upper bound of \eqref{ineq.term2}:}

Let us first introduce several quantities.
Set $\Delta(\pi,\znh,\znt) = \abs{\phi_n(\pi,\znh) - \Phi_n(\pi,\znt)}$, $\Delta^+(\pi,\znh,\znt) = \phi_n(\pi,\znh) - \Phi_n(\pi,\znt)$, and $\Delta^-(\pi,\znh,\znt) = - \Delta^+(\pi,\znh,\znt) $. 
Then, it comes
\begin{align*}
P^*\croch{ \sup_{\pi} \De > \eta}
& \leq  P^*\croch{ \sup_{\pi} \acc{\Dm\1_{\Dm> 0} } > \eta }  \\ 
& + P^*\croch{ \sup_{\pi} \acc{\Dp\1_{\Dp\geq 0}}  > \eta} \enspace.
\end{align*}
\begin{enumerate}
  \item If $\Dm> 0$, 
\begin{align*}
    \abs{ \phi_n(\pi,\znh)-\Phi_n(\pi,\znt) } & = \Phi_n(\pi,\znt) - \phi_n(\pi,\znh) \\
& \leq \Phi_n(\pi,\znt) - \phi_n(\pi,\znt) \enspace,
\end{align*}
since $\phi_n(\pi,\znt)\leq \phi_n(\pi,\znh)$. Then, Proposition~\ref{prop.unif.conv.cond.model} leads to
\begin{align*}
    \sup_{\pi} \acc{\Dm\1_{\Dm> 0} } & \leq \sup_{(\zn,\pi)\in\mathcal{P}} \Delta(\pi,\zn,\zn) \\
& \qquad \xrightarrow[n\to +\infty]{\P}0\enspace.
\end{align*}

  \item Otherwise $\Dp\geq 0$, 
\begin{align*}
    \abs{ \phi_n(\pi,\znh)-\Phi_n(\pi,\znt) } & = \phi_n(\pi,\znh) - \Phi_n(\pi,\znt) \enspace.
\end{align*}
Distinguishing between settings where $(\znh,\pi)\in\mathcal{P}$ or not, it results 
\begin{align*}
&  P^*\croch{ \sup_{\pi} \acc{\Dp\1_{\Dp\geq 0}}  > \eta} \\
& \leq  P^*\croch{ \sup_{\pi,\ (\znh,\pi) \in\mathcal{P}} \acc{\Dp\1_{\Dp\geq 0}}  > \eta} \\
& + P^*\croch{ \sup_{\pi,\ (\znh,\pi) \not\in\mathcal{P}} \acc{\Dp\1_{\Dp\geq 0}}  > \eta} \enspace.
\end{align*}
%
\subparagraph{If $(\znh,\pi)\in\mathcal{P}$:}

~\\
$\phi_n(\pi,\znh) > -\infty$ and
	 \begin{align*}
		\abs{ \phi_n(\pi,\znh)-\Phi_n(\pi,\znt) } & \leq \phi_n(\pi,\znh) - \Phi_n(\pi,\znh) \enspace
	 \end{align*}
	 by definition of $\znt$.
According to Proposition~\ref{prop.unif.conv.cond.model},
one gets
\begin{align*}
& \sup_{\pi,\ (\znh,\pi) \in\mathcal{P}} \acc{\Dp\1_{\Dp\geq 0}} \\
& \leq \sup_{(\zn,\pi)\in\mathcal{P}} \Delta(\pi,\zn,\zn)  \quad \xrightarrow[n\to +\infty]{\P}0\enspace.
\end{align*}

\subparagraph{Otherwise $(\znh,\pi)\not\in\mathcal{P}$:}

\begin{align*}
& P^*\croch{ \sup_{\pi,\ (\znh,\pi) \not\in\mathcal{P}} \acc{\Dp\1_{\Dp\geq 0}}  > \eta} \\
& \leq P^*\croch{ \exists \pi,\ \Phi_n(\pi,\znh)=-\infty,\ \Dp>\eta } \enspace.
\end{align*}
Set a sequence $\acc{\epsilon_n}_{n\in\N^*}$ such that $\epsilon_n\to 0$ and $n\epsilon_n\to +\infty$ as $n\to +\infty$. Then,
\begin{align}
& P^*\croch{ \sup_{\pi,\ (\znh,\pi) \not\in\mathcal{P}} \acc{\Dp\1_{\Dp\geq 0}}  > \eta} \nonumber \\
& = P^*\croch{ \exists \pi,\ \Phi_n(\pi,\znh)=-\infty,\ \Dp>\eta,\ N(\znh,\pi) \leq \epsilon_n n(n-1) } \label{ineq.leq}\\
& + P^*\croch{ \exists \pi,\ \Phi_n(\pi,\znh)=-\infty,\ \Dp>\eta,\ N(\znh,\pi) > \epsilon_n n(n-1) } \label{ineq.geq} \enspace,
\end{align}
where $N(\znh,\pi) = \abs{ \acc{ (i,j) \mid i\neq j,\ \pi_{\zh_i,\zh_j} \in \acc{0,1}\ \mathrm{and}\ \pi_{\zh_i,\zh_j} \neq \pi^*_{z^*_i,z^*_j} } }$.

The first term \eqref{ineq.leq} in the right-hand side is dealt with by Proposition~\ref{prop.copy.znh}:
\begin{align*}
& P^*\croch{ \exists \pi,\ \Phi_n(\pi,\znh)=-\infty,\ \Dp>\eta,\ N(\znh,\pi) \leq \epsilon_n n(n-1) } \\
& \leq  P^*\croch{ \exists \pi,\ \Phi_n(\pi,\znh)=-\infty,\ 2 a_n + \Delta(\pi,\znh^P,\znh^P)>\eta,\ N(\znh,\pi) \leq \epsilon_n n(n-1) } \\
& \leq  P^*\croch{  2 a_n + \sup_{(\zn,\pi)\in\mathcal{P}}\Delta(\pi,\zn,\zn)>\eta } \quad \xrightarrow[n\to +\infty]{}0 \enspace,
\end{align*}
following the proof of Proposition~\ref{prop.unif.conv.cond.model}.

The second term \eqref{ineq.geq} is upper bounded by noticing that
\begin{align*}
& \acc{ \phi_n(\pi,\znh)> -\infty } \cap \acc{ \Phi_n(\pi,\znh) = -\infty }\\
= &  \acc{ \sum_{(i,j)\in\M_0 } X_{i,j}=0 } \cap \acc{ \sum_{(i,j)\in\M_1 } (1-X_{i,j})=0 } 
 \enspace,
\end{align*}
where $\mathcal{M}_0 = \acc{ (i,j) \mid i\neq j,\ \pi_{\zh_i,\zh_j}=0 \ \mathrm{and}\ \pi^*_{z^*_i,z^*_j} >0 } $ and $\mathcal{M}_1 = \acc{ (i,j) \mid i\neq j,\ \pi_{\zh_i,\zh_j}=1 \ \mathrm{and}\ \pi^*_{z^*_i,z^*_j} <1 } $.\\ 
Thus,
\begin{align*}
&  P^*\croch{ \exists \pi,\ \Phi_n(\pi,\znh)=-\infty,\ \Dp>\eta,\ N(\znh,\pi) > \epsilon_n n(n-1) } \\
& \leq P^* \croch{  \sum_{k=1 }^{\epsilon_n n(n-1)} Y_k=0 } = (1-\xi)^{ \epsilon_n n(n-1)} \quad \xrightarrow[n\to +\infty]{}0 \enspace,
\end{align*}
where $\acc{Y_k}_{1\leq k \leq \epsilon_n n (n-1)}$ denote \iid Bernoulli variables with parameter $\xi = \min_{(q,l),\ \pi^*_{q,l}\not\in\acc{0,1} } \croch{\pi^*_{q,l} \wedge 1-\pi^*_{q,l} } $, and $a\wedge b = \min(a,b)$.

\end{enumerate}

Finally since no upper bound does depend on $\zn^*$, every convergence in probability with respect to $P^*$ can be turned into a convergence with respect to $\P$.

\paragraph{Upper bound of \eqref{ineq.term3}:}

$\Phi_n(\pi,\zn)$ can be expressed as:
\begin{align} \label{PhiA}
 & \Phi_n(\pi,\zn) \nonumber\\
= &  \sum_{qlq'l'}\frac{N_{qq'}(\zn)N_{ll'}(\zn)}{n(n-1)} \left[
\pi^*_{q,l}\log\pi_{q'l'}+(1-\pi^*_{q,l})\log(1-\pi_{q'l'}) \right]
 \enspace,
\end{align}
where $N_{qq'}(\zn)=\abs{ \acc{ i \mid z^*_i=q,\ \mathrm{and}\ z_i=q' } }$.

Let $\widetilde{N}_{qq'}(\pi)=N_{qq'}(\znt(\pi))$, $N^*_q= \abs{\acc{ i \mid z^*_i=q }}$, $\widetilde{a}_{qq'}(\pi)=\frac{\widetilde{N}_{qq'}(\pi)}{N^*_q}$,
and $\widetilde{A}_{\pi}$ the stochastic matrix of  $\widetilde{a}_{qq'}(\pi)$.
Coefficient $\widetilde{a}_{qq'}(\pi)$ yield the proportion of vertices from class $q$ attributed to class $q'$ by $\zn$.
Note that  \eqref{PhiA} shows that $\Phi_n(\pi,\zn)$ only depends on $\zn$ through the matrix $\widetilde{A}_{\pi}$.
Therefore, one uses the notation $\Phi_n(\pi,A(\zn))$ in place of $\Phi_n(\pi,\zn)$.

Definitions of $\widetilde{A}_{\pi}$ and $\bar{A}_{\pi}$ imply that
$\Phi_n(\pi,\widetilde{A}_{\pi}) \geq \Phi_n(\pi,\bar{A}_{\pi})$ and
$\mathbb{M}(\pi)=\mathbb{M}(\pi,\bar{A}_{\pi}) \geq \mathbb{M}(\pi,\widetilde{A}_{\pi})$.
Therefore,
\begin{enumerate}
   \item
$ \Phi_n(\pi,\widetilde{A}_{\pi}) \geq \mathbb{M}(\pi) $ \\
$\Rightarrow$ \ $0\ \le \Phi_n(\pi,\widetilde{A}_{\pi}) -
\mathbb{M}(\pi) \leq  \Phi_n(\pi,\widetilde{A}_{\pi}) -
\mathbb{M}(\pi,\widetilde{A}_{\pi})$,

   \item
$ \Phi_n(\pi,\widetilde{A}_{\pi}) \le \mathbb{M}(\pi) $ \\
$ \Rightarrow$ \ $0\ \le \mathbb{M}(\pi) -
\Phi_n(\pi,\widetilde{A}_{\pi}) \le \mathbb{M}(\pi,\bar{A}_{\pi}) -
\Phi_n(\pi,\bar{A}_{\pi}) $.
 \end{enumerate}
Then,
\begin{align*}
    \abs{ \Phi_n(\pi, \widetilde{A}_{\pi} )-\mathbb{M}(\pi) } \leq \sup_{A \in \mathcal{A}} \abs{\Phi_n(\pi, A)-\mathbb{M}(\pi,A)}\enspace.
\end{align*}
Moreover for every $A\in\mathcal{A}$,
{\small
\begin{align*}
\quad & \Phi_n(\pi,A)-\mathbb{M}(\pi,A) = \\
 &
   \sum_{qq'll'}\croch{\frac{N^*_q N^*_l}{n(n-1)} - \alpha^*_q \alpha^*_l} a_{qq'}a_{ll'}\croch{ \pi^*_{q,l}\log\pi_{q'l'}+(1-\pi^*_{q,l})\log(1-\pi_{q'l'}) } \enspace.
\end{align*}}
Since any
$\pi_{q'l'}\in\acc{0,1}$ such that
$\pi^*_{q,l} \neq  \pi_{q'l'}$ is excluded, \eqref{assum.pi.trunc} provides
\begin{align*}
\abs{ \pi^*_{q,l}\log\pi_{q'l'}+(1-\pi^*_{q,l})\log(1-\pi_{q'l'}) } \leq \Delta(\zeta) < +\infty \enspace,
\end{align*}
where $\Delta(\zeta)>0$ is independent of $\pi$ and $q$, and only
depends on $\zeta$ from Assumption \eqref{assum.pi.trunc}.

Then, since $0\leq a_{q,l}\leq 1$ for every $(q,l)$, the strong law of large numbers applied to each $N^*_q$ entails that
$\sup_{\pi}\acc{|\Phi_n\paren{\pi,\znt(\pi)} - \mathbb{M}(\pi)|} \xrightarrow[n\to +\infty]{}0 \quad \P-a.s.\enspace.$

\subsection{Proof of Proposition~\ref{prop.copy.znh}}

\begin{prop}[Existence of a copy of $\znh$ in $\mathcal{P}$]
\label{prop.copy.znh} Let $\pi$ be defined as in Section~\ref{subsec.modelSBM} and satisfying \eqref{assum.pi.trunc}.
Let us further assume that there exists a sequence $\acc{\epsilon_n}_{n\in\N^*}$ such that $\epsilon_n\to 0$ and $n\epsilon_n\to +\infty$ as $n\to +\infty$, and
\begin{align*}
\abs{ \acc{ (i,j) \mid i\neq j,\ \pi_{\zh_i,\zh_j} \in \acc{0,1}\ \mathrm{and}\ \pi_{\zh_i,\zh_j} \neq \pi^*_{z^*_i,z^*_j} } } \leq \epsilon_n n(n-1),\quad a.s. \enspace.
\end{align*}
Then, there exist $\zn^{P}\in\mathcal{P}$ and a real sequence $\acc{a_n}_{\N^*}$ such that
\begin{align*}
 0 \leq  \phi_n(\pi,\znh) - \phi_n(\pi,\zn^P) \leq a_n, \ \mathrm{and} \ \abs{\Phi_n(\pi,\znh) - \Phi_n(\pi,\zn^P)} \leq a_n,\quad a.s.\enspace,
\end{align*}
where $a_n \to 0$ as $n\to +\infty$, and $a_n$ does neither depend on $\zn^*$ nor on $\pi$.
  
\end{prop}

\medskip

\begin{proof}[Proof of Proposition~\ref{prop.copy.znh}]
First let us introduce 
\begin{align*}
L = \acc{(q_1,q_2)\in\acc{1,\ldots,Q}^2\mid N_{q_1,q_2} > n\sqrt{\epsilon_n} } \enspace,  
\end{align*}
where 
\begin{align*}
N_{q_1,q_2} = \abs{ \acc{1\leq i \leq n\mid  \zh_i=q_1,\ z^*_i=q_2 } } \enspace.
\end{align*}
For every $1\leq i \leq n$, 
we define $z_i^P$ in the following way:
\begin{enumerate}
  \item $z_i^P=\zh_i$, if $(\zh_i,z^*_i) \in L$,
  \item $z_i^P=c(z^*_i)$, otherwise,
\end{enumerate}
where $1\leq c(z^*_i)\leq Q$ is obtained by applying 
Lemma~\ref{lem.existence.large.class} with $q_2=z^*_i$.\\
Then it results that $(z_i^P,z^*_i)\in L$  for every $1\leq i\leq n$.

Let us now introduce 
\begin{align*}
  \mathcal{N} = \acc{ (q,q',l,l') \in\acc{1,\ldots,Q}^4\mid \pi_{q,l}\in\acc{0,1}\ \mathrm{and}\ \pi^*_{q',l'}\neq \pi_{q,l} } \enspace.
\end{align*}
Then for every couple $(i,j)$, $(z^P_i,z^*_i,z^P_j,z^*_j) \not\in \mathcal{N}$ since $(z^P_i,z^*_i)\in L$ and $(z^P_j,z^*_j)\in L$ thanks to Lemma~\ref{lem.quadruple.pb}.
As a consequence, it comes $\zn^P \in \mathcal{P}$ since
\begin{align*}
\acc{ (i,j) \mid i\neq j,\ \pi_{z^P_i,z^P_j} \in {0,1}\ \mathrm{and}\ \pi_{z^P_i,z^P_j} \neq \pi^*_{z^*_i,z^*_j} }  = \emptyset  \enspace.
\end{align*}

Finally, the conclusion results from \eqref{assum.pi.trunc} by noticing that the number of changes between $\znh$ and $\zn^P$ is at most $Q^2n\sqrt{\epsilon_n}$.

\end{proof}

\bigskip

\begin{lem}\label{lem.existence.large.class}
Set $L = \acc{(q_1,q_2)\in\acc{1,\ldots,Q}^2\mid N_{q_1,q_2} > n\sqrt{\epsilon_n} }$, where 
\begin{align*}
N_{q_1,q_2} = \abs{ \acc{1\leq i \leq n\mid  \zh_i=q_1,\ z^*_i=q_2 } } \enspace.
\end{align*}
With the notation and assumptions of Proposition~\ref{prop.copy.znh}, if \eqref{assum.alpha.empirique.trunc} holds true 
then 
\begin{align*}
  \forall 1\leq q_2\leq Q,\quad \exists 1\leq q_1\leq Q,\quad(q_1,q_2) \in L \enspace.
\end{align*}
\end{lem}

\medskip

\begin{proof}[Proof of Lemma~\ref{lem.existence.large.class}]
Otherwise, there exists $q_2$ such that for every $1\leq q_1 \leq Q$, $(q_1,q_2) \not\in L$.
Then,
\begin{align*}
  \abs{ \acc{1\leq i\leq n\mid z^*_i=q_2}} = \sum_{q_1=1}^Q N_{q_1,q_2} \leq Q n\sqrt{\epsilon_n} \enspace,
\end{align*}
which contradicts \eqref{assum.alpha.empirique.trunc}.
\end{proof}

\bigskip

\begin{lem}\label{lem.quadruple.pb}
  With the same notation and assumptions as Lemma~\ref{lem.existence.large.class}, let us introduce
\begin{align*}
  \mathcal{N} = \acc{ (q,q',l,l') \in\acc{1,\ldots,Q}^4\mid \pi_{q,l}\in\acc{0,1}\ \mathrm{and}\ \pi^*_{q',l'}\neq \pi_{q,l} } \enspace.
\end{align*}
Then,
\begin{align*}
  (q,q',l,l') \in\mathcal{N} \quad \Rightarrow\quad (q,q')\not\in L\ \mathrm{or}\ (l,l')\not\in L \enspace.
\end{align*}

\end{lem}

\medskip

\begin{proof}[Proof of Lemma~\ref{lem.quadruple.pb}]
  If $(q,q')\in L\ \mathrm{and}\ (l,l')\in L$, then $N_{q,q'}N_{l,l'}>n^2\epsilon_n$, which contradicts
that 
\begin{align*}
\abs{ \acc{ (i,j) \mid i\neq j,\ \pi_{\zh_i,\zh_j} \in {0,1}\ \mathrm{and}\ \pi_{\zh_i,\zh_j} \neq \pi^*_{z^*_i,z^*_j} } } \leq \epsilon_n n(n-1) \enspace.
\end{align*}
\end{proof}


\newpage

\section{Proof of Proposition~\ref{prop.distrib.conv.zn.2}}\label{append.prop.loi.Z}

\subsection{Proof of Proposition~\ref{prop.distrib.conv.zn.2}}

\paragraph{Preliminaries}

First in the same line as the proof of Theorem~\ref{thm.distrib.conv.zn}, the main 
quantity to deal with is 
\begin{align}
& \log
\frac{\widehat{P}^{\Xn}(\zn)}{\widehat{P}^{\Xn}(\zn^*)} \label{exp.log.ratio} \\
= &
   \sum_{i\neq j}
\acc{X_{i,j}\log\paren{\frac{\widehat{\pi}_{z_i,z_j}}{\widehat{\pi}_{z^*_i,z^*_j}}}
+(1-X_{i,j})\log\paren{\frac{1-\widehat{\pi}_{z_i,z_j}}{1-\widehat{\pi}_{z^*_i,z^*_j}}}}
+\sum_i\log\frac{\widehat{\alpha}_{z_i}}{\widehat{\alpha}_{z^*_i}} \enspace, \nonumber
\end{align}
where $\widehat{P}^{\Xn}(\zn)=\P_{\widehat{\alpha},\widehat{\pi}}\paren{\Zn = \zn \mid \Xn}$ is the same quantity as $P^{\Xn}(\zn)=\P_{\alpha^*,\pi^*}\paren{\Zn = \zn \mid \Xn}$ where $(\alpha^*,\pi^*)$ has been replaced by $(\widehat{\alpha},\pih)$, and $\zn$ and $\zn^*$ denote label vectors such that $\norm{\zn - \zn^*}_0 = r$, with $1\leq r\leq n$.

Second, let us assume that 
\begin{align} \label{constraint.norme.infinie}
  \norm{ \pih - \pi^*}_{\infty} \leq \min\croch{ \zeta,\  \min_{(q,l),\ \pi^*_{q,l}\not\in\acc{0,1}} \acc{\frac{\pi^*_{q,l}(1-\pi^*_{q,l})}{2}} } \enspace,
\end{align}
where $\zeta$ is given by \eqref{assum.pi.trunc}, which is fulfilled on the event 
\begin{align} \label{def.Omega.pi}
    \Omega_n = \acc{ \norm{\widehat{\pi}-\pi^*}_{\infty} \leq v_n} \enspace
\end{align}
for large enough values of $n$ since $v_n = o\paren{\sqrt{\log n}/n}$.
Note that by assumption, $\P\croch{\Omega_n^c}\xrightarrow[n\to \infty]{}0$.
It is also important to notice that the definition of $\pih$  implies that every $\pih_{q,l} \in \acc{0,1}\cup [\zeta,1-\zeta]$ (see \eqref{assum.pi.trunc}), which leads on $\Omega_n$ to 
\begin{align}\label{exp.well.defined.expression}
  \forall (q,l),\quad  \pi^*_{q,l}\in\acc{0,1} \ \Rightarrow\ \pih_{q,l} = \pi^*_{q,l} \enspace.
\end{align}

Finally for a given vector $\zn$, let us introduce the following sets of couples $(i,j)$:
\begin{align}
  D^* = D^*(\zn) & \defegal \acc{ (i,j)\mid i\neq j,\ \pi^*_{z_i,z_j}\neq \pi^*_{z^*_i,z^*_j} } \enspace, \label{exp.ensemble.ind.pistar}\\
  \Dh = \Dh(\zn) & \defegal \acc{ (i,j)\mid i\neq j,\ \pih_{z_i,z_j}\neq \pih_{z^*_i,z^*_j} } \enspace. \label{exp.ensemble.ind.pihat}
\end{align}

\paragraph{Proof}

First, the log-ratio \eqref{exp.log.ratio} can be decomposed into 
the following terms 
\begin{align*}
& \log \frac{\widehat{P}^{\Xn}(\zn)}{\widehat{P}^{\Xn}(\zn^*)} \\
= & \sum_{i\neq j}
\acc{ X_{i,j}\log\paren{\frac{\pi^*_{z_i,z_j}}{\pi^*_{z^*_i,z^*_j}}}
+(1-X_{i,j})\log\paren{\frac{1-\pi^*_{z_i,z_j}}{1-\pi^*_{z^*_i,z^*_j}}} } + \sum_i\log\frac{\widehat{\alpha}_{z_i}}{\widehat{\alpha}_{z^*_i}}  \\
& + \sum_{i\neq j}
\acc{X_{i,j}\log\paren{ \frac{\pih_{z_i,z_j}}{\pi^*_{z_i,z_j}} \frac{\pi^*_{z^*_i,z^*_j}}{\pih_{z^*_i,z^*_j}} }
+(1-X_{i,j})\log\paren{ \frac{1-\pih_{z_i,z_j}}{1-\pi^*_{z_i,z_j}} \frac{1-\pi^*_{z^*_i,z^*_j}}{1-\pih_{z^*_i,z^*_j}} } }
\enspace\cdot
\end{align*}
Second using the definition of $D^*$ and $\Dh$ given by Eq.~\eqref{exp.ensemble.ind.pistar} and Eq.~\eqref{exp.ensemble.ind.pihat}, and 
that
$ \bar D^*\cap \bar \Dh \subset  \acc{(i,j)\mid i\neq j,\ h(X_{i,j};\pih,\pi^*,z_i,z_j,z^*_i,z^*_j)=0} $
where $\bar A$ denotes the complement of any set $A$ and  
\begin{align*}
& h(X_{i,j};\pih,\pi^*,z_i,z_j,z^*_i,z^*_j) \\
=&\ X_{i,j}\log\paren{ \frac{\pih_{z_i,z_j}}{\pi^*_{z_i,z_j}} \frac{\pi^*_{z^*_i,z^*_j}}{\pih_{z^*_i,z^*_j}} }
+(1-X_{i,j})\log\paren{ \frac{1-\pih_{z_i,z_j}}{1-\pi^*_{z_i,z_j}} \frac{1-\pi^*_{z^*_i,z^*_j}}{1-\pih_{z^*_i,z^*_j}} }\enspace,  
\end{align*}
it results
\begin{align*}
& \log \frac{\widehat{P}^{\Xn}(\zn)}{\widehat{P}^{\Xn}(\zn^*)} \\
= & \sum_{(i,j)\in D^*}
\acc{ X_{i,j}\log\paren{\frac{\pi^*_{z_i,z_j}}{\pi^*_{z^*_i,z^*_j}}}
+(1-X_{i,j})\log\paren{\frac{1-\pi^*_{z_i,z_j}}{1-\pi^*_{z^*_i,z^*_j}}} } + \sum_i\log\frac{\widehat{\alpha}_{z_i}}{\widehat{\alpha}_{z^*_i}}  \\
& + \sum_{(i,j)\in D^*\cup\Dh }
\acc{X_{i,j}\log\paren{ \frac{\pih_{z_i,z_j}}{\pi^*_{z_i,z_j}} \frac{\pi^*_{z^*_i,z^*_j}}{\pih_{z^*_i,z^*_j}} }
+(1-X_{i,j})\log\paren{ \frac{1-\pih_{z_i,z_j}}{1-\pi^*_{z_i,z_j}} \frac{1-\pi^*_{z^*_i,z^*_j}}{1-\pih_{z^*_i,z^*_j}} } } \enspace.
\end{align*}
Finally from the following equalities
\begin{align*}
\log \widehat{\pi}_{z_i,z_j} &  =\log \pi^*_{z_i,z_j}+ \log \croch{
1+\frac{\widehat{\pi}_{z_i,z_j}-\pi^*_{z_i,z_j}}{\pi^*_{z_i,z_j}} } \enspace
\end{align*}
and
\begin{align*}
\log (1-\widehat{\pi}_{z_i,z_j}) &  =\log (1-\pi^*_{z_i,z_j})+ \log
\croch{
1-\frac{\widehat{\pi}_{z_i,z_j}-\pi^*_{z_i,z_j}}{1-\pi^*_{z_i,z_j}} } \enspace,
\end{align*}
the last sum can be further split into
\begin{align*}
& \sum_{(i,j)\in D^*\cup\Dh }
\acc{X_{i,j}\log\paren{ \frac{\pih_{z_i,z_j}}{\pi^*_{z_i,z_j}} \frac{\pi^*_{z^*_i,z^*_j}}{\pih_{z^*_i,z^*_j}} }
+(1-X_{i,j})\log\paren{ \frac{1-\pih_{z_i,z_j}}{1-\pi^*_{z_i,z_j}} \frac{1-\pi^*_{z^*_i,z^*_j}}{1-\pih_{z^*_i,z^*_j}} } } \\
= &   \sum_{(i,j)\in D^*\cup\Dh }  \log\croch{
1+\frac{(\widehat{\pi}_{z_i,z_j}-\pi^*_{z_i,z_j})(X_{i,j}-\pi^*_{z_i,z_j})}{\pi^*_{z_i,z_j}(1-\pi^*_{z_i,z_j})}}
\\
& -  \sum_{(i,j)\in D^*\cup\Dh } \log\croch{
1+\frac{(\widehat{\pi}_{z^*_i,z^*_j}-\pi^*_{z^*_i,z^*_j})
(X_{i,j}-\pi^*_{z^*_i,z^*_j})}{\pi^*_{z^*_i,z^*_j}(1-\pi^*_{z^*_i,z^*_j})} }\enspace.
\end{align*}
This leads to
\begin{align*}
& \log \frac{\widehat{P}^{\Xn}(\zn)}{\widehat{P}^{\Xn}(\zn^*)} \\
= & \sum_{(i,j)\in D^*}
\acc{ X_{i,j}\log\paren{\frac{\pi^*_{z_i,z_j}}{\pi^*_{z^*_i,z^*_j}}}
+(1-X_{i,j})\log\paren{\frac{1-\pi^*_{z_i,z_j}}{1-\pi^*_{z^*_i,z^*_j}}} } + \sum_i\log\frac{\widehat{\alpha}_{z_i}}{\widehat{\alpha}_{z^*_i}}  \\
& + \sum_{(i,j)\in D^*\cup\Dh }  \log\croch{
1+\frac{(\widehat{\pi}_{z_i,z_j}-\pi^*_{z_i,z_j})(X_{i,j}-\pi^*_{z_i,z_j})}{\pi^*_{z_i,z_j}(1-\pi^*_{z_i,z_j})}}
\\
& -  \sum_{(i,j)\in D^*\cup\Dh } \log\croch{
1+\frac{(\widehat{\pi}_{z^*_i,z^*_j}-\pi^*_{z^*_i,z^*_j})
(X_{i,j}-\pi^*_{z^*_i,z^*_j})}{\pi^*_{z^*_i,z^*_j}(1-\pi^*_{z^*_i,z^*_j})} } \\
= & T_1 + T_2 - T_3
\enspace.
\end{align*}
Note that \eqref{exp.well.defined.expression} implies for every $1\leq i\neq j \leq n$,
\begin{align*}
\pi^*_{z_i,z_j}\in\acc{0,1} \Rightarrow \log\croch{
1+\frac{(\widehat{\pi}_{z_i,z_j}-\pi^*_{z_i,z_j})(X_{i,j}-\pi^*_{z_i,z_j})}{\pi^*_{z_i,z_j}(1-\pi^*_{z_i,z_j})} } = 0 \enspace.
\end{align*}
In the sequel, the strategy consists in providing successive upper bounds for $T_1$, $T_2$, and $T_3$.

\paragraph{Upper bounding $T_1$}
~\\
The magnitude of $T_1$ is given by a similar argument to that in the proof of Theorem~\ref{thm.distrib.conv.zn}.
Let us consider
\begin{align*}
T_1  =&   \sum_{D^*}
\acc{X_{i,j}\log\paren{\frac{\pi^*_{z_i,z_j}}{\pi^*_{z^*_i,z^*_j}}}
+(1-X_{i,j})\log\paren{\frac{1-\pi^*_{z_i,z_j}}{1-\pi^*_{z^*_i,z^*_j}}}}
+\sum_i\log\frac{\widehat{\alpha}_{z_i}}{\widehat{\alpha}_{z^*_i}} \\
= &  \sum_{D^*}
\acc{\paren{X_{i,j}-\pi^*_{z^*_i,z^*_j}}\log\paren{\frac{\pi^*_{z_i,z_j}}{\pi^*_{z^*_i,z^*_j}}\frac{1-\pi^*_{z^*_i,z^*_j}}{1-\pi^*_{z_i,z_j}}}
 } \\
 & + \sum_{D^*}
\acc{ \pi^*_{z^*_i,z^*_j} \log\paren{\frac{\pi^*_{z_i,z_j}}{\pi^*_{z^*_i,z^*_j}}}
+(1-\pi^*_{z^*_i,z^*_j})\log\paren{\frac{1-\pi^*_{z_i,z_j}}{1-\pi^*_{z^*_i,z^*_j}}}} + \sum_i\log\frac{\widehat{\alpha}_{z_i}}{\widehat{\alpha}_{z^*_i}} \\
= & \ T_{1,1} + T_{1,2}
\enspace.
\end{align*}
Then for every $t\in\R$,
\begin{align*}
    P^*\croch{ T_1 > t } = P^*\croch{ T_{1,1}+T_{1,2} > t } \enspace.
\end{align*}

\subparagraph{Upper bound of $T_{1,2}$:}
The same proof as that of Lemma~\ref{lem.Hoeffding.Expectation} shows there exists a constant $K(\pi^*) = K^* >0$ such that
\begin{align*}
T_{1,2} \paren{\abs{ D^* }}^{-1} 
 \leq &\ \max_{  (q,l)\neq (q',l'), \pi^*_{q,l}\not \in \acc{0,1}}   - k\paren{\pi^*_{q,l},\pi^*_{q',l'}} =  - K^* <0 \enspace,
\end{align*}
for large enough values of $n$, where $ k\paren{\pi^*_{q,l},\pi^*_{q',l'}} = \pi^*_{q,l}\log\paren{\frac{\pi^*_{q,l}}{\pi^*_{q',l'}}}
+(1-\pi^*_{q,l})\log\paren{\frac{1-\pi^*_{q,l}}{1-\pi^*_{q',l'}}}$ and $\abs{ D^* }$ denotes the cardinality of $D^*$.
Thus,
\begin{align*}
    P^*\croch{ T_1 > t } \leq P^*\croch{ T_{1,1} - \abs{ D^* } K^* > t } \enspace.
\end{align*}

\subparagraph{Upper bound of $T_{1,1}$:}

\begin{align*}
 & P^*\croch{ T_1 > t } \\
 \leq &\ P^* \croch{ \sum_{(i,j)\in D^*}
\paren{X_{i,j}-\pi^*_{z^*_i,z^*_j}}\log\paren{\frac{\pi^*_{z_i,z_j}}{\pi^*_{z^*_i,z^*_j}}\frac{1-\pi^*_{z^*_i,z^*_j}}{1-\pi^*_{z_i,z_j}}}  >t  + \abs{D^*} K^* } \enspace.
\end{align*}
Hoeffding's inequality associated with \eqref{assum.pi.trunc} provides a constant $C_{\zeta}>0$ such that for every $t\in\R$
\begin{align*}
& P^* \croch{ \sum_{(i,j)\in D^*}
\paren{X_{i,j}-\pi^*_{z^*_i,z^*_j}}\log\paren{\frac{\pi^*_{z_i,z_j}}{\pi^*_{z^*_i,z^*_j}}\frac{1-\pi^*_{z^*_i,z^*_j}}{1-\pi^*_{z_i,z_j}}}  >t  + \abs{D^*} K^* } \\
\leq &  \exp \croch{- \frac{ \abs{D^*}^2 (K^*)^2 + 2t \abs{D^*} K^* }{\abs{D^*} C_{\zeta}}} = \exp\croch{-2t \frac{K^*}{C_{\zeta}}} \cdot \exp \croch{- \abs{D^*} \frac{  (K^*)^2}{C_{\zeta}} } \enspace.
\end{align*}

\paragraph{Upper bounding $T_2$}
~\\
With $t>0$ on the event $\acc{ T_2 >t}$, $\log(1+x)\leq x $ for every $x>-1$ leads to
\begin{align*}
0 < t < T_2  \leq \sum_{(i,j)\in\Dh\cup D^*} 
\frac{(\widehat{\pi}_{z_i,z_j}-\pi^*_{z_i,z_j})(X_{i,j}-\pi^*_{z_i,z_j})}{\pi^*_{z_i,z_j}(1-\pi^*_{z_i,z_j})} \enspace \cdot 
\end{align*}
Then with $N_{q',l'}^{q,l}=\sum_{(i,j)\in\Dh\cup D^*} \1_{(z_i^*=q',z_j^*=l')}\1_{(z_i=q,z_j=l)}$, it comes
\begin{align*} 
T_2 & \leq   \sum_{q,l} \abs{ \frac{(\widehat{\pi}_{q,l}-\pi^*_{q,l})}{\pi^*_{q,l}(1-\pi^*_{q,l})} \sum_{ \tiny \begin{array}{c}
                           (i,j)\in\Dh\cup D^*\\
z_i=q,\ z_j=l
                         \end{array}
}
 \paren{ X_{i,j}-\pi^*_{q,l} } } \\
& \leq   \sum_{q,l}  \abs{\frac{(\widehat{\pi}_{q,l}-\pi^*_{q,l})}{\pi^*_{q,l}(1-\pi^*_{q,l})} } \abs{\sum_{q',l'} \sum_{ \tiny \begin{array}{c}
                           (i,j)\in\Dh\cup D^* \\
(z^*_i,z^*_j)=(q',l')\\
(z_i,z_j)=(q,l)
                         \end{array}
} \paren{ X_{i,j}-\pi^*_{q',l'} } } \\
& + \sum_{q,l} \abs{ \frac{(\widehat{\pi}_{q,l}-\pi^*_{q,l})}{\pi^*_{q,l}(1-\pi^*_{q,l})} }\abs{ \sum_{q',l'}  N_{q',l'}^{q,l} \paren{ \pi^*_{q',l'}- \pi^*_{q,l} } }
\enspace.
\end{align*}

Introducing the event $\Omega_n$ defined by \eqref{def.Omega.pi} and using  \eqref{assum.pi.trunc}, one gets for every $t>0$
\begin{align*}
& P^*\croch{\Omega_n \cap \acc{T_2 > t} } \\
\leq & P^*\croch{  \zeta^2  \sum_{q,l} 
\abs{ \sum_{q',l'} \sum_{ \tiny \begin{array}{c}
                           (i,j)\in\Dh\cup D^* \\
(z^*_i,z^*_j)=(q',l')\\
(z_i,z_j)=(q,l)
                         \end{array}}
 \paren{ X_{i,j}-\pi^*_{q',l'} } }  > t/(2v_n) } \\
& + P^*\croch{ \zeta^2 \sum_{q,l}   \abs{ \sum_{q',l'} N_{q',l'}^{q,l} \paren{ \pi^*_{q',l'}- \pi^*_{q,l} } } > t/(2v_n)} 
\enspace.
\end{align*}

For the first term, Hoeffding's inequality requires summing over a deterministic set of indices, which leads to
\begin{align*}
& P^*\croch{  \zeta^2  \sum_{q,l} 
\abs{ \sum_{q',l'} \sum_{ \tiny \begin{array}{c}
                           (i,j)\in\Dh \cup D^* \\
(z^*_i,z^*_j)=(q',l')\\
(z_i,z_j)=(q,l)
                         \end{array}}
 \paren{ X_{i,j}-\pi^*_{q',l'} } }  > t/(2v_n) }  \\
= & \sum_{k} \sum_{D,\ \abs{D}=k}   P^*\croch{  \zeta^2  \sum_{q,l} 
\abs{ \sum_{q',l'} \sum_{ \tiny \begin{array}{c}
                           (i,j)\in D \cup D^*\\
(z^*_i,z^*_j)=(q',l')\\
(z_i,z_j)=(q,l)
                         \end{array}}
 \paren{ X_{i,j}-\pi^*_{q',l'} } }  > t/(2v_n) } \enspace.
\end{align*}
where the sum over $k$ is computed for $ \lceil \gamma/2\, n r\rceil \leq k \leq 2nr$ by Proposition~\ref{A2} and Lemma~\ref{lem.upper.bound.number.couples}.

For each set $D$ such that $\abs{D}=k$, a union bound and Hoeffding's inequality provide
\begin{align*}
& P^*\croch{  \zeta^2  \sum_{q,l} 
\abs{ \sum_{q',l'} \sum_{ \tiny \begin{array}{c}
                           (i,j)\in D\cup D^* \\
(z^*_i,z^*_j)=(q',l')\\
(z_i,z_j)=(q,l)
                         \end{array}}
 \paren{ X_{i,j}-\pi^*_{q',l'} } }  > t/(2v_n) } \\
\leq &  Q^2 \max_{q,l} P^*\croch{ \abs{ \sum_{q',l'} \sum_{\tiny \begin{array}{c}
(i,j) \in D\cup D^* \\
(z^*_i,z^*_j)=(q',l')\\
(z_i,z_j)=(q,l)
\end{array}}
\paren{ X_{i,j} - \pi^*_{q',l'}} }   > t/\croch{2v_n (\zeta Q)^2} }\\
 \leq & Q^2  \exp \croch{- \frac{2}{\croch{2v_n (\zeta Q)^2}^2} \frac{  t^2}{k+\abs{D^*} } } =  
Q^2  \exp \croch{- \frac{1}{ 2(\zeta Q)^4 } \frac{ t^2}{v_n^2 (k+\abs{D^*}) } }
\enspace.
\end{align*}
Then,
\begin{align*}
& P^*\croch{  \zeta^2  \sum_{q,l} 
\abs{ \sum_{q',l'} \sum_{ \tiny \begin{array}{c}
                           (i,j)\in\Dh \cup D^*\\
(z^*_i,z^*_j)=(q',l')\\
(z_i,z_j)=(q,l)
                         \end{array}}
 \paren{ X_{i,j}-\pi^*_{q',l'} } }  > t/(2v_n) }  \\
\leq & Q^2 \sum_{k=\lceil \gamma/2\, n r\rceil}^{2nr}  (2nr)^k  \exp \croch{- \frac{1}{ 2(\zeta Q)^4 } \frac{ t^2}{v_n^2 (k+\abs{D^*}) } }
\enspace.
\end{align*}

For the second term, Lemma~\ref{lem.upper.bound.number.couples} provides
\begin{align*}
& P^*\croch{ \zeta^2 \sum_{q,l} \abs{ \sum_{q',l'} N_{q',l'}^{q,l} \paren{ \pi^*_{q',l'}- \pi^*_{q,l} } } > t/ (2v_n) } \\
\leq & Q^2 \max_{q,l} P^*\croch{  \abs{ \sum_{q',l'} N_{q',l'}^{q,l} \paren{ \pi^*_{q',l'}- \pi^*_{q,l} } } >  t/ \croch{2v_n (\zeta Q)^2} } \\
\leq & Q^2 P^*\croch{ 4nr > t/ \croch{2v_n (\zeta Q)^2} } \enspace.
\end{align*}

\paragraph{Upper bounding $T3$}
~\\
Let us first notice
\begin{align*}
 & \log\croch{
1+\frac{(\widehat{\pi}_{z^*_i,z^*_j}-\pi^*_{z^*_i,z^*_j})
(X_{i,j}-\pi^*_{z^*_i,z^*_j})}{\pi^*_{z^*_i,z^*_j}(1-\pi^*_{z^*_i,z^*_j})} } \\
= & (1-X_{i,j}) \log\croch{
1-\frac{ (\widehat{\pi}_{z^*_i,z^*_j}-\pi^*_{z^*_i,z^*_j}) }{ (1-\pi^*_{z^*_i,z^*_j}) } } + X_{i,j}
\log\croch{
1+\frac{ (\widehat{\pi}_{z^*_i,z^*_j}-\pi^*_{z^*_i,z^*_j})
 }{\pi^*_{z^*_i,z^*_j} } }
\enspace\cdot
\end{align*}
Then,
\begin{align*}
 & \sum_{(i,j)\in \Dh\cup D^*} \log\croch{
1+\frac{(\widehat{\pi}_{z^*_i,z^*_j}-\pi^*_{z^*_i,z^*_j})
(X_{i,j}-\pi^*_{z^*_i,z^*_j})}{\pi^*_{z^*_i,z^*_j}(1-\pi^*_{z^*_i,z^*_j})} } \\
= & \sum_{q,l} \sum_{\tiny \begin{array}{c}
                           (i,j)\in \Dh\cup D^* \\
(z^*_i,z^*_j)=(q,l)\\
                         \end{array}}
(1-X_{i,j}) \log\croch{
1-\frac{ (\widehat{\pi}_{q,l}-\pi^*_{q,l}) }{ (1-\pi^*_{q,l}) } } + X_{i,j}
\log\croch{
1+\frac{ (\widehat{\pi}_{q,l}-\pi^*_{q,l})
 }{\pi^*_{q,l} } }
\enspace\cdot
\end{align*}
Centering the $X_{i,j}$s, it comes
\begin{align*}
 & \sum_{(i,j)\in \Dh\cup D^*} \log\croch{
1+\frac{(\widehat{\pi}_{z^*_i,z^*_j}-\pi^*_{z^*_i,z^*_j})
(X_{i,j}-\pi^*_{z^*_i,z^*_j})}{\pi^*_{z^*_i,z^*_j}(1-\pi^*_{z^*_i,z^*_j})} } \\
= &  \sum_{q,l} \sum_{\tiny \begin{array}{c}
                           (i,j)\in \Dh\cup D^*\\
(z^*_i,z^*_j)=(q,l)\\
                         \end{array}}
(\pi^*_{q,l}-X_{i,j}) \log\croch{
1-\frac{ (\widehat{\pi}_{q,l}-\pi^*_{q,l}) }{ (1-\pi^*_{q,l}) } } + (X_{i,j} - \pi^*_{q,l})
\log\croch{
1+\frac{ (\widehat{\pi}_{q,l}-\pi^*_{q,l})
 }{\pi^*_{q,l} } }  \\
 & + \sum_{q,l} \sum_{\tiny \begin{array}{c}
                           (i,j)\in \Dh\cup D^*\\
(z^*_i,z^*_j)=(q,l)\\
                         \end{array}}
(1-\pi^*_{q,l}) \log\croch{
1-\frac{ (\widehat{\pi}_{q,l}-\pi^*_{q,l}) }{ (1-\pi^*_{q,l}) } } + \pi^*_{q,l}
\log\croch{
1+\frac{ (\widehat{\pi}_{q,l}-\pi^*_{q,l})
 }{\pi^*_{q,l} } } 
\enspace,
\end{align*}
which leads to
\begin{align*}
T_3 = & \sum_{q,l} \paren{ \log\croch{
1+\frac{ (\widehat{\pi}_{q,l}-\pi^*_{q,l})
 }{\pi^*_{q,l} } } -  \log\croch{
1-\frac{ (\widehat{\pi}_{q,l}-\pi^*_{q,l}) }{ (1-\pi^*_{q,l}) } } } 
\sum_{\tiny \begin{array}{c}
                           (i,j)\in \Dh\cup D^*\\
(z^*_i,z^*_j)=(q,l)\\
                         \end{array}} (X_{i,j} - \pi^*_{q,l}) 
\\
 & + \sum_{q,l} N^*_{q,l} \croch{
(1-\pi^*_{q,l}) \log\croch{
1-\frac{ (\widehat{\pi}_{q,l}-\pi^*_{q,l}) }{ (1-\pi^*_{q,l}) } } + \pi^*_{q,l}
\log\croch{
1+\frac{ (\widehat{\pi}_{q,l}-\pi^*_{q,l})
 }{\pi^*_{q,l} } } }
\enspace,
\end{align*}
where $N^*_{q,l} = \sum_{(i,j)\in \Dh\cup D^*} \1_{(z^*_i=q,\ z^*_j=l)}$.

Second on the event $\Omega_n$,  \eqref{constraint.norme.infinie} and $\abs{\log(1+x)} \leq 2\abs{x}$ for every $x\in[-1/2,1/2]$ entail
\begin{align*}
\abs{ T_3 } \leq & 4 v_n \sum_{q,l}   
\abs{ \sum_{\tiny \begin{array}{c}
                           (i,j)\in  \Dh\cup D^*\\
(z^*_i,z^*_j)=(q,l)\\
                         \end{array}} (X_{i,j} - \pi^*_{q,l}) }
+4v_n \sum_{q,l} N^*_{q,l} 
\enspace.
\end{align*}
Then for every $t>0$,
\begin{align*}
P^*\croch{\Omega_n \cap \acc{\abs{T_3} > t} } & \leq P^*\croch{ 4 v_n \sum_{q,l}   
\abs{ \sum_{\tiny \begin{array}{c}
                           (i,j)\in  \Dh\cup D^*\\
(z^*_i,z^*_j)=(q,l)\\
                         \end{array}} (X_{i,j} - \pi^*_{q,l}) } > t/2 } \\
& + P^*\croch{ 4v_n \sum_{q,l} N^*_{q,l}  > t/2 } \enspace.
\end{align*}

Similarly to $T_2$, partitioning and Hoeffding's inequality lead to
\begin{align*}
& P^*\croch{ 4 v_n \sum_{q,l}   
\abs{ \sum_{\tiny \begin{array}{c}
                           (i,j)\in  \Dh\cup D^*\\
(z^*_i,z^*_j)=(q,l)\\
                         \end{array}} (X_{i,j} - \pi^*_{q,l}) } > t/2 }\\
& \leq  Q^2 \sum_{k=\lceil \gamma/2\, n r\rceil}^{2nr}  (2nr)^k  \exp\croch{ - \frac{2}{8^2Q^4} \frac{ t^2}{v_n^2 (k+\abs{D^*})} } 
\enspace,
\end{align*}
and Lemma~\ref{lem.upper.bound.number.couples} provides
\begin{align*}
P^*\croch{ 4v_n \sum_{q,l} N^*_{q,l}  > t/2 } \leq P^*\croch{ v_n  4nr  > t/8 } \enspace.
\end{align*}
Then, 
\begin{align*}
P^*\croch{\Omega_n \cap \acc{\abs{T_3} > t} } 
& \leq  Q^2 \sum_{k=\lceil \gamma/2\, n r\rceil}^{2nr}  (2nr)^k  \exp\croch{ - \frac{2}{8^2Q^4} \frac{ t^2}{v_n^2 (k+ \abs{D^*}) } }  \\
& \qquad + P^*\croch{ v_n  4nr  > t/8 } \enspace.
\end{align*}

\medskip

\paragraph{Gathering $T_1$-, $T_2$-, and $T_3$-upper bounds}
~\\

At the beginning the following steps are very close to those in the proof of Theorem~\ref{A2}.

For every $\epsilon>0$
\begin{align*}
& P^*\croch{
 \sum_{[\zn]\neq [\zn^*]} \frac{\widehat{P}^{\Xn}([\zn])}{\widehat{P}^{\Xn}([\zn^*])}  > \epsilon }\\
\leq &\ P^*\croch{ \acc{
 \sum_{[\zn] \neq [\zn^*]} \frac{\widehat{P}^{\Xn}([\zn])}{\widehat{P}^{\Xn}([\zn^*])}  > \epsilon} \cap \Omega_n } + P^*\croch{ \Omega_n^c } \enspace.
\end{align*}
Furthermore,
\begin{align*}
 & P^*\croch{ \acc{
 \sum_{[\zn]\neq [\zn^*]} \frac{\widehat{P}^{\Xn}([\zn])}{\widehat{P}^{\Xn}([\zn^*])}  > \epsilon} \cap \Omega_n }\\
\leq  & \sum_{r=1}^n \sum_{ \tiny 
\begin{array}{c}
  \zn\not\in [\zn^*] \\
\norm{\zn-\zn^*}_0=r
\end{array}} 
P^*\croch{ \acc{\log \frac{\widehat{P}^{\Xn}(\zn)}{\widehat{P}^{\Xn}(\zn^*)} > -(r+1)\log n -r \log Q+\log\epsilon} \cap \Omega_n }\\
\leq  & \sum_{r=1}^n \sum_{ \tiny 
\begin{array}{c}
  \zn\not\in [\zn^*] \\
\norm{\zn-\zn^*}_0=r
\end{array}}  
P^*\croch{ \acc{\log \frac{\widehat{P}^{\Xn}(\zn)}{\widehat{P}^{\Xn}(\zn^*)} > -5r\log n} \cap \Omega_n }\quad(n\geq \max\acc{Q,\epsilon^{-1}})\\
= &  \sum_{r=1}^n \sum_{  \tiny 
\begin{array}{c}
  \zn\not\in [\zn^*] \\
\norm{\zn-\zn^*}_0=r
\end{array}} 
 P^*\croch{ \acc{T_1 + T_2 - T_3 >
-5r\log n} \cap \Omega_n} \enspace.
\end{align*}
 It remains to deal with $P^*\croch{ \acc{T_1 + T_2 - T_3 >
-5r\log n} \cap \Omega_n}$:
\begin{align*}
 & P^*\croch{ \acc{T_1 + T_2 - T_3 >
-5r\log n} \cap \Omega_n} \\
\leq & \  P^*\croch{ \acc{ T_1 + T_2 - T_3 >
-5r\log n } \cap \Omega_n \cap \acc{ \abs{T_3}\leq r\log n } } \\
& \qquad + P^*\croch{ \acc{ \abs{T_3} > r\log n }\cap \Omega_n}\\
\leq & \  P^*\croch{ \acc{T_1 + T_2  >
-6r\log n} \cap \Omega_n }  + P^*\croch{ \acc{\abs{T_3} > r\log n } \cap \Omega_n } \\
\leq & \  P^*\croch{ T_1 > -7r\log n } + P^*\croch{ \acc{\abs{T_2} > r\log n}\cap \Omega_n } \\
&+ P^*\croch{ \acc{\abs{T_3} > r\log n}\cap \Omega_n } \enspace.
\end{align*}
Upper bounding $T_1$ comes from Proposition~\ref{A2} and results in
\begin{align*}
P^*\croch{ T_1 > -7r\log n } & \leq \exp\croch{ r\log n \frac{14K^*}{C_{\zeta}}} \cdot \exp \croch{- \abs{D^*} \frac{  (K^*)^2}{C_{\zeta}} } \\
& \leq \exp\croch{ r\log n \frac{14K^*}{C_{\zeta}}} \cdot \exp \croch{-  nr \frac{ \gamma (K^*)^2}{2C_{\zeta}} } 
\enspace.
\end{align*}
For $T_2$, Lemma~\ref{lem.upper.bound.number.couples} provides
\begin{align*}
P^*\croch{ \acc{\abs{T_2} > r\log n}\cap \Omega_n } & \leq  Q^2 \sum_{k=\lceil \gamma/2\, n r\rceil}^{2nr}  (2nr)^k  \exp \croch{- \frac{1}{ 2(\zeta Q)^4 } \frac{ (r\log n)^2}{4nr v_n^2  } } \\
& + Q^2 P^*\croch{ 4nr > r\log n/ \croch{2v_n (\zeta Q)^2} } \\
& \leq  Q^2   \exp\croch{8nr \log n } \cdot \exp \croch{- \frac{1}{ 2(\zeta Q)^4 } \frac{ (r\log n)^2}{4nr v_n^2  } } \\
& + Q^2 P^*\croch{ v_n > \frac{\log n}{ 8n (\zeta Q)^2 } } \enspace.
\end{align*}
Similarly for $T_3$, it results
\begin{align*}
  P^*\croch{ \acc{\abs{T_3} > r\log n}\cap \Omega_n } & \leq 
 Q^2 \sum_{k=\lceil \gamma/2\, n r\rceil}^{2nr}  (2nr)^k  \exp\croch{ - \frac{2}{8^2Q^4} \frac{ (r\log n)^2}{ 4nr v_n^2  } }  \\
& \qquad + P^*\croch{ v_n  4nr  > (r\log n)/8  } \\
& \leq 
 Q^2   \exp\croch{8nr \log n } \cdot \exp\croch{ - \frac{2}{8^2Q^4} \frac{ (r\log n)^2}{ 4nr v_n^2 } }  \\
& \qquad + P^*\croch{ v_n    >  \frac{ \log n}{32n}  } \enspace.
\end{align*}

From the previous bounds, one observes that requiring $v_n=o(\log n /n)$ makes $P^*\croch{ v_n    >  \frac{ \log n}{32n}  }$ and $P^*\croch{ v_n > \frac{\log n}{ 8n (\zeta Q)^2 } }$ vanish as $n$ grows, which leads to
\begin{align*}
 & P^*\croch{ \acc{T_1 + T_2 - T_3 > -5r\log n } \cap \Omega_n} \\
& \quad \leq  Q^2   \exp\croch{8nr \log n } \cdot \exp\croch{ - \frac{2}{8^2Q^4} \frac{ r(\log n)^2}{ 4n v_n^2 } } \\
& \quad  +  Q^2   \exp\croch{8nr \log n } \cdot \exp \croch{- \frac{1}{ 2(\zeta Q)^4 } \frac{ r(\log n)^2}{4n v_n^2  } } \\
& \quad  + \exp\croch{ r\log n \frac{14K^*}{C_{\zeta}}} \cdot \exp \croch{-  nr \frac{ \gamma (K^*)^2}{2C_{\zeta}} } \\
& \quad  \leq C_1 \paren{ \exp\croch{8n \log  n - C_2 \frac{(\log n)^2}{nv_n^2}} }^r
\end{align*}
for large enough values of $n$ and constants $C_1,C_2>0$ only depending on $Q$, $\zeta$, $\gamma$, and $K^*$ but not of $z^*$.

Following the same line as in the proof of Theorem~\ref{thm.distrib.conv.zn}, 
for every $\epsilon>0$ and large enough values of $n$, it comes 
\begin{align*}
 & P^*\croch{ \acc{
 \sum_{[\zn]\neq [\zn^*]} \frac{\widehat{P}^{\Xn}([\zn])}{\widehat{P}^{\Xn}([\zn^*])}  > \epsilon} \cap \Omega_n } \\
& \leq \sum_{r=1}^n {n\choose r} (Q-1)^r  C_1 \paren{ \exp\croch{8n \log  n - C_2 \frac{(\log n)^2}{nv_n^2}} }^r \\
& = C_1 \paren{ \croch{ 1+(Q-1) u_n^{\prime} }^n -1 } \enspace,
\end{align*}
where $u_n^{\prime} = \exp\croch{8n \log  n - C_2 \frac{(\log n)^2}{nv_n^2}}$.
Thus requiring $v_n = o\paren{\sqrt{\log n}/n}$, it comes
\begin{align*}
\croch{ 1+(Q-1) u_n^{\prime} }^n & = \exp\croch{ n \log\paren{ 1+ (Q-1) u_n^{\prime} } } \\
& \leq \exp\croch{ (Q-1) n u_n^{\prime} } \xrightarrow[n\to +\infty]{} 1 \enspace,
\end{align*}
which concludes the proof since no upper bound does depend on $z^*$.

\subsection{Lemma~\ref{lem.upper.bound.number.couples} }

\begin{lem}\label{lem.upper.bound.number.couples}
  Let $\pi\in\M_Q(\R)$ denote a matrix with coefficients $\pi_{q,l}$ belong to $[0,1]$, and $\zn$and $\zn^*$ be two label vectors such that $\sum_{i=1}^n \1_{z_i\neq z^*_i}=r$.
Then,
\begin{align*}
  \abs{ \acc{(i,j)\mid i\neq j,\ \pi_{z_i,z_j}\neq \pi_{z^*_i,z^*_j} } } \leq 2nr \enspace.
\end{align*}

\end{lem}

\medskip

\begin{proof}[Proof of Lemma~\ref{lem.upper.bound.number.couples}]

Without loss of generality, one can assume the first $r$ coordinates of $\zn$ are different from those of $\zn^*$.
Then, any difference between $\pi_{z_i,z_j}$ and $\pi_{_{z^*_i,z^*_j}}$ can only occur if  $(z_i,z_j)\neq (z^*_i,z^*_j)$.
It results
\begin{align*}
\abs{ \acc{(i,j)\mid i\neq j,\ \pi_{z_i,z_j}\neq \pi_{z^*_i,z^*_j} } } 
=  & \abs{ \acc{(i,j)\mid i\neq j,\ \pi_{z_i,z_j}\neq \pi_{z^*_i,z^*_j},\ i\leq r} } \\
&\ + \abs{ \acc{(i,j)\mid i,\ \pi_{z^*_i,z_j}\neq \pi_{z^*_i,z^*_j},\ i> r,\ j\leq r} } \\
\leq &\quad nr + (n-r)r \\
\leq & \quad 2nr \enspace. 
\end{align*}

\end{proof}

\bigskip

\begin{lem}\label{lem.mixture.model.class.freq}
With the same notation as Proposition~\ref{prop.distrib.conv.zn.2} and the assumptions of Theorem~\ref{thm.consist.alpha.MLE}, the maximum likelihood estimator of $\alpha$ is given by  
\begin{align*}
\forall 1\leq q \leq Q,\qquad    \widehat{\alpha}_q = \frac{1}{n}\sum_{i=1}^n \widehat{P}(Z_i=q \mid \Xn )  \enspace.
\end{align*}
\end{lem}

\medskip

\begin{proof}[Proof of Lemma~\ref{lem.mixture.model.class.freq}]
Let us introduce some notation:
\begin{itemize}
  \item  $(\zn, \pi ) \mapsto f_{\Xn}(\zn, \pi )=\L_1\paren{\Xn;\zn,\pi}$,

  \item $(\alpha, \pi ) \mapsto f_{\Xn}(\alpha, \pi )=\L_2\paren{\Xn;\alpha,\pi}$,

  \item $\alpha \mapsto f_{\Zn}(\alpha)=\prod_{i=1}^n \alpha_{Z_i}$,

  \item $(\alpha, \pi ) \mapsto f_{\Xn,\Zn}(\alpha, \pi ) = f_{\Xn}(\Zn, \pi ) f_{\Zn}(\alpha) $ denote the complete likelihood of $(\alpha,\pi)$.
  
\end{itemize}

We start computing the derivative of  $ f_{\Xn}(\alpha, \pi ) + \lambda (\sum_{q}\alpha_q-1) $ with respect to $\alpha_k$, for $1\leq k\leq Q$ and $ \lambda\in\R$. 
\begin{align*}
  \frac{ \partial\croch{ f_{\Xn}(\alpha, \pi ) + \lambda (\sum_{q}\alpha_q-1) } }{\partial\alpha_k} 
=   \sum_{\zn}  \frac{ N_k(\zn) }{\alpha_k} f_{\Xn}(\zn, \pi ) f_{\zn}(\alpha)+\lambda  \enspace,
\end{align*}
where $N_k(\zn) = \sum_{i=1}^n\1_{(z_i=k)}$.
Multiplying by $\alpha_k$ and summing over $k$ leads to
\begin{align*}
  \lambda = - n f_{\Xn}(\widehat{\alpha}, \pi ) \enspace,
\end{align*}
where $\widehat{\alpha}$ denotes the optimum location of $\alpha$ (for which the derivative vanishes).
It results for every $k$
\begin{align*}
\widehat{\alpha}_k  & = \sum_{\zn}  \frac{ N_k(\zn) }{n}  \frac{ f_{\Xn}(\zn, \pi ) f_{\Zn}(\widehat{\alpha}) }{ f_{\Xn}(\widehat{\alpha}, \pi ) } \\ 
& = \sum_{\zn}  \frac{ N_k(\zn) }{n}  \frac{ f_{\Xn,\zn}(\widehat{\alpha}, \pi )}{ f_{\Xn}(\widehat{\alpha}, \pi ) } \\ 
& = \sum_{\zn}  \frac{ N_k(\zn) }{n}  \frac{ f_{\Xn,\zn}(\widehat{\alpha}, \pi )}{ f_{\Xn}(\widehat{\alpha}, \pi ) } 
= \sum_{\zn}  \frac{ N_k(\zn) }{n}  f_{\zn}^{\Xn}(\widehat{\alpha}, \pi )
 \enspace,
\end{align*}
where $f_{\zn}^{\Xn}(\widehat{\alpha}, \pi ) = \P_{\widehat{\alpha}, \pi}\croch{ \Zn = \zn \mid \Xn }$ (Section~\ref{subsec.main.asymptotic.result}) denotes the \emph{a posteriori} probability of $\Zn=\zn$ given $\Xn$ with parameters $(\widehat{\alpha},\pi)$. 

Finally, the result comes from
\begin{align*}
  \widehat{\alpha}_k  
& = \frac{ 1 }{n}  \sum_{i=1}^n \sum_{\zn}  \1_{(z_i=k)} f_{\zn}^{\Xn}(\widehat{\alpha}, \pi ) \\
& = \frac{ 1 }{n}  \sum_{i=1}^n \sum_{\zn}  \1_{(z_i=k)} \P_{\widehat{\alpha}, \pi}\croch{ \Zn = \zn \mid \Xn } \\
& = \frac{ 1 }{n}  \sum_{i=1}^n  \P_{\widehat{\alpha}, \pi}\croch{ Z_i = k \mid \Xn } \enspace.
\end{align*}
Replacing $\pi$ by the MLE $\pih$ of $\pi^*$, the MLE of $\alpha^*$ satisfies for every $k$
\begin{align*}
  \widehat{\alpha}_k  
= \frac{ 1 }{n}  \sum_{i=1}^n  \P_{\widehat{\alpha}, \pih}\croch{ Z_i = k \mid \Xn } = \frac{ 1 }{n}  \sum_{i=1}^n  \widehat{P}\paren{ Z_i=k \mid
\Xn} \enspace.
\end{align*}

\end{proof}

\newpage

\section{Proof of Theorem~\ref{thm.uniform.P-as.Var.Likelihood}}

\begin{lem}\label{lem.unif.conv.var}
Let $\znh =\znh (\pi)=\argmax_{\zn}\L_1(\Xn; \zn, \pi)$. For every $\Xn\in\Xdefn$, $(\alpha,\pi) \in \Theta$, and $\taun \in S_n$, it comes that
\begin{align*}
\J(\Xn; \taun, \alpha, \pi) \leq \L_2(\Xn;  \alpha, \pi) \leq  \L_1(\Xn; \znh,\pi) \enspace.
\end{align*}
\end{lem}

\medskip

\begin{proof}[Proof of Lemma~\ref{lem.unif.conv.var}]
The first inequality results from the definition of $\J$ given by Eq.~\eqref{eq.variational.approx}.

The second one comes from $\znh (\pi)=\argmax_{\zn}\L_1(\Xn; \zn, \pi)$.
Thus for every $(\alpha,\pi)$,
\begin{align*}
\L_2(\Xn;  \alpha, \pi) &
 \leq  \log \left\{ e^{ \L_1(\Xn; \znh,\pi) }  \sum_{z_{[n]} \in \Zdefn}  P_{\Zn}(\zn) \right\} \leq  \L_1(\Xn; \znh,\pi) \enspace.
 \end{align*}
\end{proof}

\bigskip

\begin{lem}\label{lem.unif.gap.likelihoods.variation}
Lemma~\ref{lem.unif.conv.var} and Assumption  \eqref{assum.alpha.trunc} entail that there exists $0<\gamma<1$ such that
for every $(\alpha,\pi)$,
\begin{align*}
\abs{\L_2(\Xn; \alpha, \pi)-\L_1(\Xn; \znh, \pi)} \leq &\ n \log (1/\gamma)\enspace,\\
\abs{\J(\Xn; \tahn, \alpha, \pi)-\L_1(\Xn; \znh, \pi)} \leq &\ n \log (1/\gamma) \enspace.
\end{align*}
\end{lem}

\medskip

\begin{proof}[Proof of Lemma~\ref{lem.unif.gap.likelihoods.variation}]
From Lemma~\ref{lem.unif.conv.var} and definition of $\tahn$ it comes for every $(\alpha,\pi)$:
\begin{align*}
\J(\Xn; \znh, \alpha, \pi) \le  \J(\Xn; \tahn, \alpha, \pi) \le \L_2(\Xn; \alpha, \pi) \le  \L_1(\Xn; \znh, \pi)\enspace.
\end{align*}
Combined with
$\J(\Xn; \znh, \alpha, \pi)= \L_1(\Xn; \znh, \pi)+\sum_{i=1}^n \log \alpha_{\zh_i}$ (see Lemma~\ref{lem.Daudin.et.al}), it leads to both
\begin{align*}
\abs{\L_2(\Xn;  \alpha, \pi)-\L_1(\Xn; \znh,\pi)} \leq &\ -\sum_{i=1}^n \log \alpha_{\zh_i}\enspace,\\
\abs{\J(\Xn; \tahn, \alpha, \pi)-\L_1(\Xn; \znh, \pi)} \leq &\ -\sum_{i=1}^n \log \alpha_{\zh_i}\enspace.
\end{align*}
Assumption \eqref{assum.alpha.trunc} yields the conclusion.
\end{proof}

\bigskip

\begin{lem}\label{lem.Daudin.et.al}
With the same notation as Theorem~\ref{thm.uniform.P-as.Var.Likelihood}, let $\znh = \znh(\pi) = \argmin_{\zn} \L_1(\Xn; \zn, \pi)$. 
Then for every $(\alpha,\pi)$, 
  \begin{align*}
    \J(\Xn; \znh, \alpha, \pi)= \L_1(\Xn; \znh, \pi) +\sum_{i=1}^n \log\alpha_{\widehat{z}_i} \enspace.
  \end{align*}

\end{lem}

\medskip

\begin{proof}[Proof of Lemma~\ref{lem.Daudin.et.al}]
  First, let us recall Eq.~\eqref{eq.variational.approx}
\begin{align*}
     \J(\Xn; \taun, \alpha, \pi) & = \L_2(\Xn;  \alpha, \pi) - K\paren{D_{\taun},P^{\Xn}} \\
      & = \log \croch{f(\Xn; \alpha, \pi ) } - K\paren{D_{\taun},P^{\Xn}}\enspace,
\end{align*}
where $(\alpha,\pi) \mapsto f(\Xn; \alpha, \pi )$ denotes the likelihood of $(\alpha,\pi)$.

Second, Eq.~\eqref{def.product.distribution} and simple calculations lead to
\begin{align*}
&  K\paren{D_{\taun},P^{\Xn}} \\
& = \sum_{\zn} D_{\taun}(\zn) \log D_{\taun}(\zn)-\sum_{\zn} D_{\taun}(\zn) \log \P\paren{\Zn=\zn \mid \Xn}\\
      & = \sum_{iq} \tau_{i,q} \log \tau_{i,q}-\sum_{\zn} D_{\taun}(\zn) \log \paren{ \frac{ f(\Xn, \zn; \alpha, \pi ) }{ f(\Xn; \alpha, \pi ) } } \enspace,
\end{align*}
where $(\alpha, \pi ) \mapsto f(\Xn, \zn; \alpha, \pi )$ denotes the complete-likelihood of $(\alpha, \pi )$.
Then, 
\begin{align*}
& K\paren{D_{\taun},P^{\Xn}} - \sum_{i,q} \tau_{i,q} \log \tau_{i,q} \\
= &  -\sum_{\zn} D_{\taun}(\zn) \log f(\Xn; \zn, \pi ) \\
& + \sum_{\zn} D_{\taun}(\zn) \log f(\zn; \alpha)+\log f(\Xn; \alpha, \pi ) \enspace,
\end{align*}
where $\alpha \mapsto f(\zn; \alpha) = \prod_{i=1}^n \paren{ \sum_{q=1}^Q\alpha_q^{z_i} }$.
Hence,
\begin{align*}
& K\paren{D_{\taun},P^{\Xn}}  - \sum_{i,q} \tau_{i,q} \log \tau_{i,q} \\
       & = - \sum_{i  \neq  j} \sum_{q,l} \croch{ X_{i,j}\log \pi_{q,l}+ (1-X_{i,j})\log(1-\pi_{q,l}) } \tau_{i,q} \tau_{j,l} \\
& + \sum_{i,q}\tau_{i,q}\log \alpha_q+ \L_2(\Xn;  \alpha, \pi) \enspace.
\end{align*}
Therefore for every $\taun$,
\begin{align*}
& \J(\Xn; \taun, \alpha, \pi) \\
& = \sum_{i  \neq  j} \sum_{q,l} [X_{i,j}\log \pi_{q,l}+ (1-X_{i,j})\log(1-\pi_{q,l})] \tau_{i,q} \tau_{j,l} - \sum_{iq} \tau_{i,q} \paren{\log \tau_{i,q} -  \log\alpha_{q}} \enspace.
\end{align*}
With $\taun = \znh$, it comes
$\J(\Xn; \znh, \alpha, \pi)= \L_1(\Xn; \znh, \pi) +\sum_{i=1}^n \log\alpha_{\widehat{z}_i}$, which concludes the proof.
\end{proof}

\bigskip

\begin{lem}\label{lem.aposteriori.new.param}
\begin{align*}
    \abs{ D_{\widetilde{\tau}_{[n]}}(\zn^*) - \widetilde{P}(\zn^*)} \leq \sqrt{-\frac{1}{2} \log\croch{ \widetilde{P}(\zn^*)}}\enspace.
\end{align*}

\end{lem}

\medskip

\begin{proof}[Proof of Lemma~\ref{lem.aposteriori.new.param}]
\begin{align*}
    \abs{ D_{\widetilde{\tau}_{[n]}}(\zn^*) - \widetilde{P}(\zn^*)} & \leq \norm{ D_{\widetilde{\tau}_{[n]}} - \widetilde{P} }_{TV}\\
& \leq \sqrt{\frac{1}{2} K\paren{D_{\widetilde{\tau}_{[n]}}, \widetilde{P}}}\\
& \leq \sqrt{\frac{1}{2} K\paren{ \delta_{\zn^*}, \widetilde{P}}} = \sqrt{-\frac{1}{2} \log\croch{ \widetilde{P}(\zn^*) } } \enspace.
\end{align*}

\end{proof}


\newpage


\bibliographystyle{natbib}

\bibliography{biblio}

\end{document}